\documentclass{amsart}

\newcommand{\apref}[3]{\hyperref[#2]{#1\ref*{#2}#3}}

\usepackage{enumerate}
\usepackage[latin1]{inputenc}
\usepackage{dsfont}
\usepackage{amssymb,amsthm,amsmath}

\input{xy}
\xyoption{all}

\usepackage{mathrsfs}

\theoremstyle{plain}
\newtheorem{prop}{Proposition}[section]
\newtheorem{lemma}[prop]{Lemma}

\newtheorem{thm}[prop]{Theorem}

\newtheorem*{thmAnn}{Theorem A}
\newtheorem*{thmBnn}{Theorem B}

\newtheorem{cor}[prop]{Corollary}

\theoremstyle{definition}

\theoremstyle{remark}
\newtheorem{remark}[prop]{Remark}

\newcommand{\as}{\text{\rm as}}
\newcommand{\hol}{\text{\rm hol}}

\newcommand{\verklein}{.7}
\newcommand{\verschieb}{-.5cm}
\newcommand{\verschiebt}{-.3cm}

\newcommand{\slow}{\text{\rm slow}}
\newcommand{\fast}{\text{\rm fast}}

\newcommand{\TO}{\mc L}

\DeclareMathOperator{\diam}{diam}
\DeclareMathOperator{\len}{len}
\DeclareMathOperator{\s}{s}

\DeclareMathOperator{\sign}{sign}

\DeclareMathOperator{\Gen}{Gen}
\newcommand{\redu}{\text{red}}
\newcommand{\reg}{\text{reg}}

\DeclareMathOperator{\GL}{GL}

\DeclareMathOperator{\SL}{SL}
\DeclareMathOperator{\PSL}{PSL}
\DeclareMathOperator{\PGL}{PGL}

\DeclareMathOperator{\PSO}{PSO}

\DeclareMathOperator{\Unit}{U}

\DeclareMathOperator{\diag}{diag}

\DeclareMathOperator{\Tr}{Tr}
\DeclareMathOperator{\tr}{tr}

\DeclareMathOperator{\Ima}{Im}
\DeclareMathOperator{\Rea}{Re}

\DeclareMathOperator{\pr}{pr}

\DeclareMathOperator{\Ind}{Ind}

\DeclareMathOperator{\Stab}{Stab}

\newcommand{\dec}{\text{dec}}

\newcommand\N{\mathbb{N}}

\newcommand\R{\mathbb{R}}
\newcommand\Z{\mathbb{Z}}
\newcommand\C{\mathbb{C}}

\newcommand{\h}{\mathbb{H}}

\newcommand{\mc}[1]{\mathcal #1}

\newcommand{\wt}{\widetilde}
\newcommand{\wh}{\widehat}

\newcommand{\eps}{\varepsilon}

\DeclareMathOperator{\dvol}{dvol}

\DeclareMathOperator{\SFE}{SEF}
\DeclareMathOperator{\FFE}{FEF}

\DeclareMathOperator{\id}{id}

\DeclareMathOperator{\Fct}{Fct}
\DeclareMathOperator{\MCF}{MCF}

\newcommand{\sceq}{\mathrel{\mathop:}=}

\newcommand{\mat}[4]{\begin{pmatrix} #1&#2\\#3&#4\end{pmatrix}}
\newcommand{\bmat}[4]{\begin{bmatrix} #1&#2\\#3&#4\end{bmatrix}}
\newcommand{\textmat}[4]{\left(\begin{smallmatrix} #1&#2 \\ #3&#4
\end{smallmatrix}\right)}
\newcommand{\textbmat}[4]{\left[\begin{smallmatrix} #1&#2 \\ #3&#4
\end{smallmatrix}\right]}

\usepackage[colorlinks,breaklinks]{hyperref}
\usepackage{graphics}
\usepackage[T1]{fontenc} 
\usepackage{pdflscape}
\allowdisplaybreaks
\usepackage{mathtools}
\usepackage{epstopdf}

\makeatletter
\def\subsubsection{\@startsection{subsubsection}{3}%
  \z@{.5\linespacing\@plus.7\linespacing}{-.5em}%
  {\normalfont\bfseries}}
\makeatother

\begin{document}

\title[A transfer-operator-based relation]{A transfer-operator-based relation between Laplace eigenfunctions and zeros of Selberg zeta functions}
\author[A.\@ Adam]{Alexander Adam}
\address{AA: Institut de Math\'ematiques de Jussieu - Paris Rive Gauche, Sorbonne Universit\'e, Campus Pierre et Marie Curie, 4, place Jussieu, Boite Courrier 247 - 75252 Paris Cedex 05, France}
\email{alexander.adam@imj-prg.fr}
\author[A.\@ Pohl]{Anke Pohl}
\address{AP: University of Bremen, Department 3 -- Mathematics, Bibliothekstr.\@ 
5,  28359 Bremen, Germany}
\email{apohl@uni-bremen.de}
\begin{abstract} 
Over the last few years Pohl (partly jointly with coauthors) developed dual `slow/fast' transfer operator approaches to automorphic functions, resonances, and Selberg zeta functions for a certain class of hyperbolic surfaces $\Gamma\backslash\h$ with cusps and all finite-dimensional unitary representations $\chi$ of $\Gamma$.

The eigenfunctions with eigenvalue $1$ of the fast transfer operators determine the zeros of the Selberg zeta function for $(\Gamma,\chi)$. Further, if $\Gamma$ is cofinite and $\chi$ is the trivial one-dimensional representation then highly regular eigenfunctions with eigenvalue $1$ of the slow transfer operators characterize Maass cusp forms for $\Gamma$. Conjecturally, this characterization extends to more general automorphic functions as well as to residues at resonances. 

In this article we study, without relying on Selberg theory, the relation between the eigenspaces of these two types of transfer operators for any Hecke triangle surface $\Gamma\backslash\h$ of finite or infinite area and any finite-dimensional unitary representation $\chi$ of the Hecke triangle group $\Gamma$. In particular we provide explicit isomorphisms between relevant subspaces. This solves a conjecture by M\"oller and Pohl, characterizes some of the zeros of the Selberg zeta functions independently of the Selberg trace formula, and supports the previously mentioned conjectures.
\end{abstract}
\maketitle

\section{Introduction}

Let $\h = \PSL_2(\R)/\PSO(2)$ denote the hyperbolic plane, let $\Gamma$ be a Fuchsian group (that is, a discrete subgroup of $\PSL_2(\R)$), and let $\chi\colon \Gamma\to \Unit(V)$ be a unitary representation of $\Gamma$ on a finite-dimensional complex vector space $V$. The relation between the geometric and the spectral properties of $X\sceq \Gamma\backslash\h$ (e.\,g., volume, periodic geodesics, etc., among the geometric objects; $L^2$-eigenvalues, resonances, $(\Gamma,\chi)$-automorphic functions, etc., among the spectral entities) is an important subject with a long, rich history and ongoing high-level activity. Among the various tools and methods used in the study of this relation, one is the Selberg zeta function, another one the development of transfer operator techniques.

The Selberg zeta function establishes such a relation on the level of spectra, namely between the primitive geodesic length spectrum among the geometric properties and the Laplace resonances (i.\,e., the $L^2$-spectral parameters and the scattering resonances) among the spectral properties of $X$. More precisely, it follows (for $X$ of infinite area at least for the case of $\chi$ being the trivial character) from the properties of the Selberg zeta function that the primitive geodesic length spectrum of $X$ and the resonances of the Laplacian on $X$ determine each other. 

By its very nature, the Selberg zeta function cannot provide any such relation beyond the spectral level (unless additional information is used). This means in particular that it is not possible to construct an $L^2$-eigenfunction of the Laplacian or a resonant state using only (geometric) information provided by the properties of the Selberg zeta function. 

The modular surface $\PSL_2(\Z)\backslash\h$ was the first hyperbolic orbifold for which transfer operator techniques allowed to show a relation between the geodesic flow and Laplace eigenfunctions beyond the spectral level. 

More precisely, the combination of the articles \cite{Artin, Series, Mayer_thermo, Mayer_thermoPSL, Efrat_spectral, Lewis_Zagier, Bruggeman, Chang_Mayer_transop} shows that the even and odd Maass cusp forms for $\PSL_2(\Z)$ are isomorphic to the eigenfunctions with eigenvalue $+ 1$ and $-1$, respectively, of Mayer's transfer operator
\begin{equation}\label{TOMayer}
 \TO_s^{\text{\rm Mayer}}f(z) = \sum_{n\in\N} \frac{1}{(z+n)^{2s}} f\left(\frac{1}{z+n}\right).
\end{equation}
The transfer operator arises purely from a discretization and symbolic dynamics for the geodesic flow on $\PSL_2(\Z)\backslash\h$. Thus, this isomorphism provides a purely geometric characterization of the Maass cusp forms for $\PSL_2(\Z)$, not only of their eigenvalues or spectral parameters. Hence, these transfer operator results indeed establish a relation between geometric and spectral entities of $X$ beyond the spectral level.

The results in \cite{Artin, Series, Mayer_thermo, Mayer_thermoPSL, Efrat_spectral, Lewis_Zagier, Bruggeman, Chang_Mayer_transop} include dynamical interpretations also for other parts of the spectrum \cite{Chang_Mayer_period, Chang_Mayer_transop, Lewis_Zagier} as well as a representation of the Selberg zeta function as a Fredholm determinant of $\pm\TO_s^{\text{\rm Mayer}}$. A generalization to certain finite index subgroups of $\PSL_2(\Z)$ were achieved in \cite{Chang_Mayer_transop, Deitmar_Hilgert, Fraczek_Mayer_Muehlenbruch}. An alternative characterization of the Maass cusp forms for $\PSL_2(\Z)$ by means of a transfer operator deriving from a discretization of the geodesic flow on $\PSL_2(\Z)\backslash\h$ is provided by the combination \cite{Mayer_Stroemberg, Bruggeman_Muehlenbruch, Mayer_Muehlenbruch_Stroemberg}.

Until 2009, analogous characterizations of Maass cusp forms (or any other $L^2$-eigen\-func\-tions or resonant states) could not be achieved for any other hyperbolic orbifold $\Gamma\backslash\h$. Only the following result, of a weaker and less precise nature, could be established: For a large class of Fuchsian groups $\Gamma$, a transfer operator family $\TO_s$ ($s\in\C\setminus\{\text{poles}\}$) was found whose Fredholm determinant represents the Selberg zeta function of $\Gamma$, sometimes only up to certain correction functions \cite{Fried_zetafunctionsI, Pollicott, Fried_triangle, Morita_transfer, Patterson_Perry, Guillope_Lin_Zworski, Mayer_Muehlenbruch_Stroemberg}. Taking advantage of the spectral interpretation of the zeros of the Selberg zeta function (proved, e.\,g., by means of the Selberg trace formula) immediately implies that the eigenspaces with eigenvalue $1$ of these transfer operators are in some relation to the Maass cusp forms. This result however is only a dimension statement if at all (Jordan blocks may occur); it does not provide an insightful isomorphism (see the more detailed discussion below).

For some of the transfer operators developed for Hecke triangle groups it could even be shown that the eigenfunctions with eigenvalue $1$ are solutions of certain functional equations with finitely many terms  \cite{Mayer_Muehlenbruch_Stroemberg}, an important step towards developing an analogue of the results for $\PSL_2(\Z)$. However, to this day, these solutions could not been shown to be indeed period functions (unless $\PSL_2(\Z)$ is considered). In other words, an isomorphism between Maass cusp forms and solutions of these functional equations is still missing.

Nevertheless, such transfer operator approaches to Selberg zeta functions proved to be helpful in the study of resonances and more. As a few examples we name the results on resonance counting \cite{Guillope_Lin_Zworski} and location \cite{Naud_resonancefree}, the numerical studies of the structure of the set of resonances for Fuchsian Schottky groups \cite{Borthwick_numerical, Borthwick_Weich}, the numerical and rigorous studies of the behavior of zeros of the Selberg zeta function under perturbations \cite{Fraczek_diss, Fraczek_Mayer, Bruggeman_Fraczek_Mayer}, the progress towards Zaremba's conjecture \cite{Bourgain_Kontorovich} and the generalization of Selberg's $3/16$ Theorem \cite{Bourgain_Gamburd_Sarnak}. We refer to \cite{Pohl_hecke_infinite} for more examples.

It is reasonable to expect that a deeper understanding of the relation between the geometry of a hyperbolic orbifold $\Gamma\backslash\h$, its automorphic functions and resonant states, and its Selberg zeta functions allow us to prove even deeper results. We refer to \cite{Anantharaman_Zelditch, Bettin_Conrey} where the aforementioned deeper results for $\PSL_2(\Z)$ are used. 

The results in this article are a further step towards such a deeper understanding. We remark that the results presented in this article do not make any use of the Selberg trace formula or scattering theory. Therefore they provide a proper alternative, complement or extension of the relations obtained with these other methods.

The articles \cite{Pohl_diss, Hilgert_Pohl, Pohl_mcf_general, Pohl_mcf_Gamma0p, Moeller_Pohl, Pohl_Symdyn2d, Pohl_hecke_infinite, Pohl_spectral_hecke, Pohl_representation} document part of a recent program to systematically develop dual `slow/fast' transfer operator approaches to automorphic functions, resonances and Selberg zeta functions for a certain class of (cofinite and non-cofinite) Fuchsian groups $\Gamma$ with cusps. 

\begin{figure}[h]
\xymatrix{
&\fbox{\begin{minipage}{1.7cm} geod.\@ flow\\ on $X$ \end{minipage}}\ar[dl]_{\begin{minipage}{1.9cm}\begin{center}{\footnotesize slow discre- tization}\end{center}\end{minipage}} \ar[dr]^{\quad\begin{minipage}{1.8cm}\begin{center}{\footnotesize fast discre-tization}\end{center}\end{minipage}}
\\
\fbox{
\begin{minipage}{3.85cm}
slow (`finite-term') tra\-ns\-fer oper\-ators $\TO_s^\slow$
\end{minipage}\ar[d]
\ar@{<-->}[rr]^{?}
}
&&
\fbox{
\begin{minipage}{3.55cm}
fast (`infinite-term')\\ tra\-ns\-fer operators $\TO_s^\fast$
\end{minipage}\ar[d]
}
\\
\fbox{
\begin{minipage}{3.85cm}
$\{ f = \TO_s^\slow f\} \cong \MCF_s$; 
\\
conjecture on automorphic cusp forms;
\\
conj.\@ on resonances
\end{minipage}
}
&&
\fbox{
\begin{minipage}{3.55cm}
$Z(s) = \det\left(1-\TO_s^\fast\right)$
\end{minipage}
}
}
\caption{Dual transfer operator approaches}
\end{figure}\label{fig:program}

A rough schematic overview of the structure of these transfer operator approaches is given in Figure~\ref{fig:program}. We refer to Section~\ref{sec:prelims} below for more details. In Figure~\ref{fig:program}, all entities may depend on $X=\Gamma\backslash\h$. The function $Z=Z_{\Gamma,\chi}$ denotes the Selberg zeta function of $(\Gamma,\chi)$, and $\MCF_s$ denotes the space of Maass cusp forms for $\Gamma$ with spectral parameter $s$. 

Further, `slow' refers to the property that each point of the discrete dynamical system used in the definition of the `slow' transfer operators has finitely many preimages only, or equivalently, that the symbolic dynamics arising from the discretization of the geodesic flow on $X$ uses a finite alphabet only (see \cite{Pohl_diss, Pohl_Symdyn2d}).  Hence, `slow' transfer operators involve finite sums only. In contrast, `fast' means that points with infinitely but countably many preimages occur, and hence the associated `fast' transfer operators involve infinite sums.  The fast discretizations arise from the slow ones by a certain induction or acceleration process (which also explains the naming). We refer to \cite{Moeller_Pohl, Pohl_hecke_infinite, Pohl_representation} for details. 

The discretizations and the transfer operators developed within this program are typically different from those in the articles mentioned above. An exception are the fast discretization and fast transfer operator for the modular group $\PSL_2(\Z)$ which coincide essentially with the ones in \cite{Artin, Series} and \cite{Mayer_thermo, Mayer_thermoPSL}, respectively.

We refer to Section~\ref{sec:proof} below for examples of the transfer operators developed within this program. Further, we refer to the articles \cite{Pohl_diss, Hilgert_Pohl, Pohl_mcf_general, Pohl_mcf_Gamma0p, Moeller_Pohl, Pohl_Symdyn2d, Pohl_hecke_infinite, Pohl_spectral_hecke, Pohl_representation} and the references therein for a more comprehensive exposition of such transfer operator approaches, their history and their relation to mathematical quantum chaos and other areas, and remain here rather brief.

If $\chi$ is the trivial one-dimensional representation and $\Gamma$ is cofinite (and admissible for these techniques) then the slow transfer operators $\TO_s^\slow$ provide a dynamical characterization of the Maass cusp forms for $\Gamma$ \cite{Pohl_mcf_general}. More precisely, for $s\in\C$, $\Rea s\in (0,1)$, the Maass cusp forms with spectral parameter $s$  are isomorphic to the eigenfunctions of the transfer operator $\TO^\slow_s$ with eigenvalue $1$ of sufficient regularity (`period functions'). The proof of the isomorphism between Maass cusp forms and these period functions takes advantage of the characterization of Maass cusp forms in parabolic cohomology as provided by \cite{BLZ_part2}. Both, \cite{Pohl_mcf_general} and \cite{BLZ_part2} do not rely on the Selberg trace formula, any scattering theory, or the Selberg zeta function.

For general finite-dimensional unitary representations $\chi$ and general admissible Fuchsian groups $\Gamma$ it is expected that the sufficiently regular eigenfunctions with eigenvalue $1$ of $\TO^\slow_s$ characterize $(\Gamma,\chi)$-automorphic functions or are closely related to the residue operator at the resonance $s$ \cite{Pohl_hecke_infinite, Pohl_representation}.

The fast operators $\TO^\fast_s$ are nuclear operators of order $0$ that represent the Selberg zeta function $Z_{\Gamma,\chi}$ of $\Gamma$ as a Fredholm determinant:
\[
 Z_{\Gamma,\chi}(s) = \det\left(1-\TO^\fast_s\right).
\]
Hence the zeros of $Z_{\Gamma,\chi}$ are determined by the eigenfunctions of $\TO^\fast_s$ with eigenvalue $1$ \cite{Moeller_Pohl, Pohl_hecke_infinite, Pohl_spectral_hecke,  Pohl_representation}. Also this proof is independent of the Selberg trace formula and of geometric scattering theory. 

For several combinations of $(\Gamma,\chi)$ (e.g., if $\Gamma$ is any cofinite geometrically finite, non-elementary Fuchsian group or if $\chi$ is the trivial character and $\Gamma$ is geometrically finite, non-elementary) Selberg theory, geometric scattering theory or microlocal analysis allows to show a relation between (some of) the zeros of $Z_{\Gamma}$ and the spectral parameters of the Maass cusp forms for $\Gamma$ or $(\Gamma,\chi)$-automorphic forms and, more generally, the resonances of $\Delta$ on $\Gamma\backslash\h$. Hence it provides a link (on the spectral level) between the two bottom objects in Figure~\ref{fig:program}.

It is natural to ask if this relation derives as a shadow of a link between the geodesic flow and certain spectral entities beyond the spectral level. In other words, the question arises if and how these spectral entities can be explicitly characterized as eigenfunctions with eigenvalue $1$ of the fast transfer operator $\TO_s^\fast$.

In order to simplify the discussion of the nature of this question we restrict---for a moment---to the case that $\Gamma$ is a lattice (that is, $\Gamma$ is cofinite \cite[Definition~1.8]{Raghunathan}), $\chi$ the trivial character and to Maass cusp forms as the spectral entities of interest.

Selberg theory in combination with functional analysis for nuclear operators of low orders on Banach spaces allows us to deduce only a rather weak version of such a link. We may only conclude that some, rather unspecified subspaces of eigenfunctions of $\TO_s^\fast$ are isomorphic to some, rather unspecified subspaces of Maass cusp forms (or period functions and hence certain eigenfunctions of $\TO_s^\slow$). At the current state of art, neither Selberg theory nor any other non-transfer operator based approach provides us with a tool to answer any of the following questions:
\begin{enumerate}[(i)]
\setlength{\itemsep}{1mm}
\item\label{Q1} How can we characterize these subspaces of eigenfunctions of $\TO_s^\fast$, how the subspaces of Maass cusp forms? 
\item Is there an insightful isomorphism between these subspaces?
\item The zeros of Selberg zeta functions do not only consist of the spectral parameters of Maass cusp forms but also of scattering resonances and topological zeros. All of these zeros are detected by eigenfunctions with eigenvalue $1$ of $\TO_s^\fast$. Which additional properties of these eigenfunctions are needed in order to distinguish the spectral parameters of Maass cusp forms from scattering resonances?
\item\label{Q4} The transfer operator $\TO_s^\fast$ may have Jordan blocks with eigenvalue $1$. The order of $s$ as a zero of the Selberg zeta functions corresponds to the algebraic multiplicity (hence the size of the Jordan blocks), not necessarily the geometric multiplicity of $1$ as an eigenvalue of $\TO_s^\fast$. Further, $s$ as a spectral parameter for Maass cusp forms may have a higher multiplicity. In such a case, are the dimensions of the $1$-eigenspace of $\TO_s^\fast$ (considered as acting on which space?) and the space of the Maass cusp forms equal? If not, does the transfer operator detect only some of the Maass cusp forms?
\end{enumerate}

In this article we show that---purely within the framework of transfer operators---we are able to provide such a link between the geodesic flow and certain spectral entities beyond the spectral level and to answer questions in \eqref{Q1}-\eqref{Q4} at least for the case of Maass cusp forms. Moreover, we lay the groundwork for the generalization to other spectral entities as well. Their complete characterization in terms of eigenfunctions of $\TO_s^\fast$ has to await their characterization in terms of eigenfunctions of $\TO_s^\slow$.

The full details for the construction of fast transfer operators $\TO_s^\fast$ are up to now provided for (cofinite and non-cofinite) Hecke triangle groups only. Anyhow, the structure of these constructions clearly applies to a wider class of Fuchsian groups. 

However, also in this article we focus on the family of Hecke triangle groups and show that the $1$-eigenspaces of the slow and fast transfer operators are indeed isomorphic (the dotted `?'-arrow in Figure~\ref{fig:program}) as conjectured in \cite{Moeller_Pohl, Pohl_hecke_infinite, Pohl_representation}.

\begin{thmAnn}
Let $\Gamma$ be a (cofinite or non-cofinite) Hecke triangle group and $\chi$ a finite-dimensional unitary representation of $\Gamma$, and let $\Rea s > 0$. Suppose that $\TO_s^\slow$ and $\TO_s^\fast$ are the associated families of slow and fast transfer operators, respectively. Then the eigenfunctions with eigenvalue $1$ of $\TO_s^\fast$ are isomorphic to the real-analytic eigenfunctions with eigenvalue $1$ of $\TO_s^\slow$ that satisfy a certain growth restriction.
\end{thmAnn}

The isomorphism in Theorem~A is explicit and constructive. Moreover, if $\Gamma$ is a lattice and $\chi$ is the trivial one-dimensional representation then the period functions (i.\,e., those eigenfunctions of $\TO_s^\slow$ that are isomorphic to the Maass cusp forms for $\Gamma$ with spectral parameter $s$) can be characterized as a certain subspace of the eigenfunctions of $\TO_s^\fast$. More generally, additional conditions of a certain type on the eigenfunctions of $\TO_s^\slow$ translate to essentially the same conditions on the eigenfunctions of $\TO_s^\fast$. We refer to Theorems~\ref{thm:main_finite}, \ref{thm:main_theta} and \ref{thm:main_nonco} below for more details. 

Neither the proof of Theorem~A nor the characterization of the subspace of eigenfunctions of $\TO_s^\fast$ that corresponds to period functions---and hence Maass cusp forms---uses Selberg theory. Therefore these results allow us to classify some of the zeros of the Selberg zeta function purely within this transfer operator framework and independently of the use of a Selberg trace formula.

Theorem~A, more precisely Theorems~\ref{thm:main_finite}, \ref{thm:main_theta} and \ref{thm:main_nonco} below in combination with the characterization of Maass cusp forms as eigenfunctions of the slow transfer operators $\TO_s^\slow$, yields answers to the questions in \eqref{Q1}-\eqref{Q4} and provides, for Hecke triangle groups other than $\PSL_2(\Z)$, the first result of this kind. As already mentioned, for the case that $\Gamma=\PSL_2(\Z)$ and that $\chi$ is the trivial one-dimensional representation even more is known due to the combination of \cite{Bruggeman, Chang_Mayer_period, Chang_Mayer_transop, Lewis_Zagier, Deitmar_Hilgert}. We comment on it in more details in Section~\ref{remarks} below.

The restriction to Hecke triangle groups allows us to actually prove a stronger statement than Theorem~A. Each Hecke triangle group commutes with a certain element $Q\in\PGL_2(\R)$ of order $2$, which acts as an orientation-reversing Riemannian isometry on $\h$. This exterior symmetry is compatible with the transfer operators, and hence induces their splitting into odd parts $\TO_s^{\slow,-}$ and $\TO_s^{\fast,-}$ as well as even parts $\TO_s^{\slow,+}$ and $\TO_s^{\fast,+}$. 

If $\Gamma$ is cofinite, $\chi$ is the trivial character and $\Rea s \in (0,1)$ then the sufficiently regular eigenfunctions with eigenvalue $1$ of $\TO_s^{\slow,+}$ (equivalently, the eigenfunctions with eigenvalue $1$ of $\TO_s^\slow$ that are invariant under the action of $Q$) are isomorphic to the even Maass cusp forms for $\Gamma$. Likewise, the eigenfunctions with eigenvalue $1$ of $\TO_s^{\slow,-}$ (equivalently, the eigenfunctions with eigenvalue $1$ of $\TO_s^\slow$ that are anti-invariant under the action of $Q$) are isomorphic to the odd Maass cusp forms for $\Gamma$ \cite{Moeller_Pohl, Pohl_spectral_hecke}. The Fredholm determinant of the transfer operator family $\TO_s^{\fast,+}$ equals the Selberg-type zeta function whose zeros encode the even part of the spectrum of $\Gamma$, and the Fredholm determinant of $\TO_s^{\fast,-}$ equals the Selberg-type zeta function of the odd part of the spectrum of $\Gamma$ \cite{Pohl_spectral_hecke}.

Instead of Theorem~A we show its strengthend version that considers separately the odd and even transfer operators.

\begin{thmBnn}
Let $\Gamma$ be a (cofinite or non-cofinite) Hecke triangle group, $\chi$ a finite-dimensional unitary representation of $\Gamma$, and $\Rea s > 0$, and suppose that $\TO_s^{\slow,\pm}$ and $\TO_s^{\fast,\pm}$ are the associated families of slow/fast even/odd transfer operators. Then the real-analytic eigenfunctions with eigenvalue $1$ of $\TO_s^{\slow,+}$ (respectively  $\TO_s^{\slow,-}$) that satisfy a certain growth condition are isomorphic to the eigenfunctions with eigenvalue $1$ of $\TO_s^{\fast,+}$ (respectively  $\TO_s^{\fast,-}$).
\end{thmBnn}

The same comments as for Theorem~A apply to Theorem~B. In particular, the isomorphism in Theorem~B is explicit and constructive, and certain additional conditions on eigenfunctions can be included. Therefore, even and odd Maass cusp forms can be characterized as certain eigenfunctions of $\TO_s^{\fast,\pm}$, respectively. Again we refer to Theorems~\ref{thm:main_finite}, \ref{thm:main_theta} and \ref{thm:main_nonco} below for precise statements.

Moreover, Theorems~A and B support the conjectures on the significance of the eigenfunctions of $\TO_s^\slow$ in Figure~\ref{fig:program}. In addition, Patterson \cite{Patterson_israel} proposed a cohomological framework for the  divisors of Selberg zeta functions. 
If $\Gamma$ is a lattice and $\chi$ is the trivial one-dimensional representation then---as mentioned above---certain eigenspaces of $\TO_s^\slow$ for the eigenvalue $1$ are isomorphic to parabolic $1$-cohomology spaces, and hence Theorems~A and B support Patterson's conjecture.  We discuss this further in Section~\ref{remarks} below.

In Section~\ref{sec:prelims} below we provide the necessary background on Hecke triangle groups and transfer operators. In Section~\ref{sec:proof} below we prove Theorems~A and B, and in the final Section~\ref{remarks} below we briefly comment on the underlying structure of the isomorphism maps for Theorems~A and B, and the possibility for their generalizations. 

The Appendix is not needed for the understanding of the proofs of Theorems~A and B. It should be seen as background information on part of the motivation. It provides a sketch of a proof of the splitting of the Selberg zeta function according to the splitting of the transfer operators which is not worked out yet in the literature for all combinations $(\Gamma,\chi)$ that we consider throughout this article.

\textit{Acknowledgement.} The authors wish to thank the Centro di Ricerca Matematica Ennio di Giorgi in Pisa for the warm hospitality where part of this work was conducted. The second-named author AP wishes to thank the Max Planck Institute for Mathematics in Bonn for financial support and excellent working conditions during the preparation of this manuscript. Further, AP acknowledges support by the DFG grant PO 1483/2-1. Finally, the authors wish to thank the referee for thorough reading and many comments that improved the exposition.

\section{Preliminaries}\label{sec:prelims}

\subsection{The hyperbolic plane} As a model for the hyperbolic plane we use the upper half plane
\[
 \h \sceq \{ z\in\C \mid \Ima z > 0\}
\]
endowed with the well-known hyperbolic Riemannian metric given by the line element
\[
 ds^2 = \frac{dz d\overline{z}}{(\Ima z)^2}. 
\]
We identify its geodesic boundary with $P^1(\R) \cong \R \cup \{\infty\}$. The action of the group of Riemannian isometries on $\h$ extends continuously to $P^1(\R)$.

This group of isometries is isomorphic to 
\[
G\sceq \PGL_2(\R) = \GL_2(\R)/ (\R^\times\cdot\id),
\]
its subgroup of orientation-preserving Riemannian isometries is 
\[
\PSL_2(\R) = \SL_2(\R)/\{\pm\id\}.
\]
The action of $\PSL_2(\R)$ on $\h\cup P^1(\R)$ is given by fractional linear transformations, i.\,e., for $\textbmat{a}{b}{c}{d}\in \PSL_2(\R)$ and $z\in\h\cup\R$ we have
\[
\bmat{a}{b}{c}{d}.z = 
\begin{cases}
\frac{az+b}{cz+d} & \text{for $cz+d\not=0$}
\\
\infty & \text{for $cz+d=0$} 
\end{cases}
\quad\text{and}\quad
\bmat{a}{b}{c}{d}.\infty = 
\begin{cases}
\frac{a}{c} & \text{for $c\not=0$}
\\
\infty & \text{for $c=0$.}
\end{cases}
\]

\subsection{Hecke triangle groups}

The Hecke triangle group $\Gamma_\ell$ with parameter $\ell>0$ is the subgroup of $\PSL_2(\R)$ generated by the two elements
\begin{equation}\label{generators}
 S\sceq \bmat{0}{1}{-1}{0}\quad\text{and}\quad  T_\ell \sceq \bmat{1}{\ell}{0}{1}.
\end{equation}
It is Fuchsian if and only if $\ell \geq  2$ or $\ell = 2\cos \frac{\pi}{q}$ with $q\in\N_{\geq 3}$. In the following, the expression `Hecke triangle group' always refers to a Fuchsian Hecke triangle group, and we refer to the spaces $X_\ell = \Gamma_\ell\backslash\h$
as \textit{Hecke triangle surfaces}.

The (Fuchsian) Hecke triangle groups form a $1$-parameter subgroup of Fuchsian groups which contains both arithmetic and non-arithmetic groups as well as groups of finite co-area as well as group of infinite co-area. Moreover, it contains the well-studied modular subgroup $\PSL_2(\Z)$ (for $\ell=1$, that is, $q=3$). We provide a few more details about these groups. 

\begin{figure}[h]
\begin{center}
\includegraphics{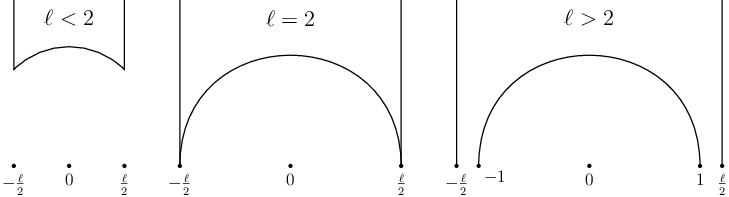}
\end{center}
\caption{Fundamental domain for $\Gamma_\ell$.}
\label{funddoms}
\end{figure}

A fundamental domain for the Hecke triangle group $\Gamma_\ell$ is given by (see Figure~\ref{funddoms})
\[
 \mc F_\ell \sceq \{ z\in\h \mid |z|>1,\ |\Rea z|<\ell/2 \}.
\]
The side-pairings for $\mc F_\ell$ are provided by the generators \eqref{generators}: the vertical sides $\{ \Rea z = -\ell/2\}$ and $\{ \Rea z = \ell/2\}$ are identified via $T_\ell$, and the two bottom sides $\{ |z|=1,\, \Rea z\leq 0\}$ and $\{|z|=1,\, \Rea z\geq 0\}$ are identified via $S$.

Among the Hecke triangle groups those and only those with parameters $\ell\leq 2$ are lattices, i.\,e., are cofinite. For 
\[
\ell = \ell(q) =  2\cos\frac{\pi}{q}
\]
with $q\in\N_{\geq 3}$, the Hecke triangle surface $X_\ell$ has a single cusp (represented by $\infty$) and two elliptic points. In the special case $q=3$, thus $\ell(q)=1$,  the Hecke triangle group $\Gamma_1$ is the modular group $\PSL_2(\Z)$. 

The Hecke triangle group $\Gamma_2$ is commonly known as the Theta group. It is conjugate to the Hecke congruence subgroup $\Gamma_0(2)$, more precisely to its image in $\PGL_2(\R)$. The associated Hecke triangle surface $X_2$ has two cusps (represented by $\infty$ and $\ell/2$) and one elliptic point. 

The Hecke triangle groups $\Gamma_\ell$ with $\ell \in \{ \ell(3), \ell(4), \ell(6), 2\}$ are the only arithmetic ones.  

For $\ell >2$, the groups $\Gamma_{\ell}$ are non-cofinite, and the orbifold $X_\ell$ has one funnel (represented by the subset $[-\ell/2,-1] \cup (1,\ell/2)$ of $\R$), one cusp (represented by $\infty$) and one elliptic point.

\subsection{Associated triangle groups and representations}

Let $\Gamma$ be a Hecke triangle group, and let 
\[
 \wt\Gamma \sceq \left\langle \Gamma, Q\right\rangle,
\]
where
\[
 Q \sceq \bmat{0}{1}{1}{0}.
\]
The group $\wt\Gamma$ is a triangle group (thus, generated by the reflections across the sides of a hyperbolic triangle), and $\Gamma$ is its subgroup of orientation-preserving isometries. Hence $\Gamma$ has index $2$ in $\wt\Gamma$.

Let $X\sceq \Gamma\backslash\h$ denote the associated Hecke triangle surface. Let $\chi$ be a finite-dimensional unitary representation of $\wt\Gamma$ on a complex vector space $V$. We consider $\chi$ to be fixed throughout.

There are many examples for finite-dimensional unitary representations $\chi$ of $\wt\Gamma$. In the following we provide a few rather explicit ones.

\begin{enumerate}[{\rm (i)}]
\item If $\Lambda$ is a subgroup of $\wt\Gamma$ of finite index and $\eta\colon\Lambda\to U(V)$ is a finite-dimensional unitary representation of $\Lambda$ then its induction $\chi = \Ind_{\Lambda}^{\wt\Gamma}\eta$ to $\wt\Gamma$ is a finite-dimensional unitary representation of $\wt\Gamma$. This construction applies in particular if $\eta$ is the trivial one-dimensional representation $\eta\sceq {\bf 1}\colon\Lambda\to \C^\times$ of $\Lambda$. 
\\\\
In addition, the choice $\eta = {\bf 1}$ allows us to understand all arising transfer operators as transfer operators for $\Lambda$ instead of twisted or $\chi$-weighted transfer operators for $\Gamma$. Thus, Theorems~A and B have further interpretations.
\item We can use presentations of $\wt\Gamma$ to construct examples of finite-dimensional unitary representations $\chi\colon\wt\Gamma\to U(V)$. To provide such examples let 
\[
 J\sceq \bmat{-1}{0}{0}{1},\quad W\sceq JT = \bmat{1}{\ell}{0}{-1},\quad U\sceq TS = \bmat{\ell}{-1}{1}{0}.
\]
\begin{enumerate}[(a)]
\item\label{choicea} 
If $\ell>2$ then a presentation of $\wt\Gamma_\ell$ is given by 
\[
 \wt\Gamma_\ell = \left\langle J, Q, W \ \left\vert\  J^2 = Q^2 = W^2 = 1,\ JQ=QJ \right.\right\rangle.
\]
Clearly, $\chi$ is well-defined and completely determined if we prescribe $\chi$ on the elements $J$, $Q$ and $W$ obeying the restrictions
\[
 \chi(J)^2 = \chi(Q)^2 = \chi(W)^2 = \id\quad\text{and}\quad \chi(J)\chi(Q) = \chi(Q)\chi(J).
\]
For example, we can set $\chi(J) = \id$ and pick any elements in $U(V)$ of order $2$ for $\chi(Q)$ and $\chi(W)$. These elements can be chosen non-trivial. E.\,g., if $V=\R^2$ then we can choose 
\[
 \chi(Q) = \mat{2}{-1}{3}{-2}\quad\text{and}\quad \chi(W) = \mat{1}{0}{0}{-1}.
\]
Obviously, many other possibilities for $\chi$ exist.
\item\label{choiceb} 
If $\ell = 2$ then a presentation of $\wt\Gamma_\ell = \wt\Gamma_2$ is given by
\[
 \wt\Gamma_2 = \left\langle S, J, T \ \left\vert\  S^2 = J^2 = (SJ)^2 = (TJ)^2 = 1\right.\right\rangle.
\]
We can construct finite-dimensional non-trivial unitary representations $\chi$ of $\wt\Gamma_2$ as in \eqref{choicea}.
\item\label{choicec} 
Finally, if $\ell =2\cos(\pi/q)$ with $q\in\N_{\geq 3}$ then a presentation of $\wt\Gamma_\ell$ is given by
\[
 \wt\Gamma_\ell = \left\langle S, Q, U \ \left\vert\   S^2 = Q^2 = (UQ)^2 =(QS)^2 =U^q = 1\right.\right\rangle.
\]
We can easily construct finite-dimensional non-trivial unitary representations $\chi\colon\wt\Gamma_\ell\to U(V)$ by setting $\chi(U) = \id$ and then proceeding as in \eqref{choicea}. Of course, also other possibilities exist. For the case $V=\C^2$ and $q=4$ we can, e.\,g., set
\[
 \chi(U) = \mat{0}{i}{i}{0},\quad \chi(Q) = \mat{-1}{0}{0}{1}, \quad \chi(S) = \mat{1}{0}{0}{1}.
\]
\end{enumerate}
\item For Theorem~A only finite-dimensional unitary representations $\chi$ of $\Gamma$ (not necessarily extendable to $\wt\Gamma$) are requested. If $\ell=2\cos(\pi/q)$ with $q\in\N_{\geq 3}$ then a presentation of $\Gamma_\ell$ is given by
\[
 \Gamma_\ell = \left\langle S, U \ \left\vert\  S^2  =U^q = 1\right.\right\rangle.
\]
An example for a non-trivial finite-dimensional representation $\chi\colon\Gamma_\ell\to\nobreak U(V)$ is, e.\,g., given as follows: Let $n\sceq \dim V$. For $j=1,\ldots, n$ pick $a_j\in \{\pm 1\}$ and let $b_j$ be a $q$-th root of unity. Then
\[
 \chi(S) \sceq \diag(a_1,\ldots, a_n)\quad\text{and}\quad \chi(U) \sceq \diag(b_1,\ldots, b_n)
\]
determines a unitary representation which is non-trivial as soon as at least one of the $a_j$ or $b_j$ is not $1$.
\end{enumerate}

\subsection{Automorphic functions, and Selberg zeta functions}\label{sec:SZF}

We say that a function $f\colon\h\to\nobreak\Gamma$ is \textit{$(\Gamma,\chi)$-automorphic} if 
\[
 f(\gamma.z) = \chi(\gamma)f(z)
\]
for all $z\in\h$, $\gamma\in\Gamma$. Let $C^\infty(X;V;\chi)$ be the space of smooth ($C^\infty$) $(\Gamma,\chi)$-automorphic functions $f$ whose restriction $f\vert_{\mc F}$ to some fundamental domain $\mc F$ for $\Gamma$ is bounded, and let $C^\infty_c(X;V;\chi)$ be its subspace of functions $f$ which satisfy that $f\vert_{\mc F}$ is compactly supported. We endow $C^\infty_c(X;V;\chi)$ with the inner product
\begin{equation}\label{innerproduct}
 (f_1,f_2) \sceq \int_{\mc F} \langle f_1(z), f_2(z)\rangle \dvol(z) \qquad \big( f_1,f_2\in C^\infty_c(X;V;\chi) \big)
\end{equation}
where $\langle\cdot,\cdot\rangle$ is the inner product on $V$, and $\dvol$ is the hyperbolic volume form. The representation $\chi$ being unitary yields that the definitions of $C^\infty(X;V;\chi)$, $C^\infty_c(X;V;\chi)$ and the inner product $(\cdot,\cdot)$ defined in \eqref{innerproduct} do not depend on the choice of $\mc F$. Let
\[
 \mc H\sceq L^2(X;V;\chi)
\]
denote the completion of $C^\infty_c(X;V;\chi)$ with respect to $(\cdot,\cdot)$. Then the Laplace-Beltrami operator
\[
 \Delta = - y^2\big(\partial_x^2 + \partial_y^2\big)
\]
on $X$ extends uniquely from 
\[
 \big\{ f\in C^\infty(X;V;\chi) \ \big\vert\ \text{$f$ and $\Delta f$ are bounded on $\mc F$}\big\}
\]
to a self-adjoint nonnegative definite operator on $\mc H$, which we also denote by $\Delta = \Delta(\Gamma;\chi)$. If $f\in\mc H$ is an eigenfunction of $\Delta$, say $\Delta f= \mu f$, we branch its eigenvalue as $\mu=s(1-s)$ and call $s$ its \textit{spectral parameter}.

The eigenfunctions of $\Delta$ in $\mc H$ that decay rapidly towards any cusp of $X$ are called \textit{cusp (vector) forms}. More precisely, for every parabolic element $p\in\Gamma$ let
\[
 V_p \sceq \big\{ v\in V \ \big\vert\ \chi(p)v=v \big\}
\]
be the subspace of $V$ consisting of the vectors fixed by the representation $\chi$ restricted to the subgroup 
\[
 \Gamma_p \sceq \{ p^n \mid n\in\Z \},
\]
and let 
\[
N_p \sceq \{ p^t \mid t\in\R\}
\]
denote the horocycle subgroup associated to $p$, thus, the one-parameter subgroup of $\PSL_2(\R)$ containing $\Gamma_p$. Then $f\in\mc H$ is called a \textit{$(\Gamma,\chi)$-cusp form} if $f$ is an eigenfunction of $\Delta$ and satisfies
\[
 \int_{\Gamma_p\backslash N_p} \langle f(z), v\rangle\, {\rm d} z = 0
\]
for all $v\in V_p$ and all parabolic $p\in\Gamma$. The measure ${\rm d} z$ here refers to the uniform measure on horocycles. 

A cusp form $f$ is called \textit{odd} if $f(-\overline z) = - f(z)$. It is called \textit{even} if $f(-\overline z) = f(z)$. If the representation $\chi$ is the trivial character then cusp forms are called \textit{Maass cusp forms}.

In order to define the Selberg zeta function for $(\Gamma,\chi)$ we recall that an element $g\in\Gamma$ is called ($\Gamma$-)\textit{primitive} if $g=h^n$ for $(h,n)\in\Gamma\times\N$ implies $n=1$ or $g=\id$. For $g\in\Gamma$ let $[g]_\Gamma$ denote its conjugacy class in $\Gamma$. Further let $[\Gamma]_p$ denote the set of all conjugacy classes of primitive hyperbolic elements in $\Gamma$. Finally, for hyperbolic $h\in\Gamma$  let $N(h)$ denote its norm, that is the square of its eigenvalue with the largest absolute value. The $\Gamma$-conjugacy classes of the primitive elements in $\Gamma$ correspond to the primitive (i.\,e., considered to be traced out once; in other words, with minimal period as length) periodic geodesics on the Hecke triangle surface $\Gamma\backslash\h$. The length of the primitive periodic geodesic $\gamma$ associated to $[g]_\Gamma\in [\Gamma]_p$ is $\ell(\gamma) = \log N(g)$.

For $\Rea s \gg 1$, the Selberg zeta function for $(\Gamma,\chi)$ is then defined by
\begin{equation}\label{infprod}
 Z(s) \sceq Z(s,\chi) \sceq \prod_{[h]_\Gamma\in [\Gamma]_p} \prod_{k=0}^\infty \det\left( 1 - \chi(h) N(h)^{-(s+k)}\right).
\end{equation}
More precisely, the abscissa of convergence of this infinite product equals the Hausdorff dimension $\delta\sceq \dim_H\Lambda(\Gamma)$ of the limit set $\Lambda(\Gamma)$ of $\Gamma$. If $\Gamma$ is cofinite then $\delta=1$, for non-cofinite $\Gamma$ we have $\delta<1$ (see, e.\,g.\@ \cite{Selberg, Patterson, Sullivan}). It is well-known that \eqref{infprod} has a meromorphic continuation to all of $\C$.

An element $h\in\wt\Gamma$ is called \textit{hyperbolic} if $h^2\in\Gamma$ is hyperbolic. Suppose that $h\in\wt\Gamma$ is hyperbolic. The norm of $h$ is defined as $N(h)=N(h^2)^{1/2}$. The element $h$ is called ($\wt\Gamma$-)\textit{primitive} if it is not a nontrivial integral power of any hyperbolic element in $\wt\Gamma$. Let $[h]_{\wt\Gamma}$ denote the $\wt\Gamma$-conjugacy class of $h$, and let $[\wt\Gamma]_p$ denote the set of $\wt\Gamma$-conjugacy classes of the $\wt\Gamma$-primitive elements in $\wt\Gamma$. 

For $\Gamma=\Gamma_\ell$ with $\ell\geq 2$ or $\ell=2\cos\tfrac{\pi}{q}$ with $q\in\N_{\geq 3}$ odd, the $\wt\Gamma$-conjugacy classes of \mbox{$\wt\Gamma$-}primitive elements in $\wt\Gamma$ correspond to the primitive periodic billiards on the triangle surface $\wt\Gamma_\ell\backslash\h$. In this case, for $\Rea s \gg 1$, the even ($+$) and odd ($-$) Selberg(-type) zeta functions are defined by
\begin{align*}
 Z_+(s) &\sceq Z_+(s,\chi) \sceq \prod_{[g]_{\wt\Gamma}\in [\wt\Gamma]_p} \prod_{k=0}^\infty \det\left( 1 - \det g^k\cdot \chi(g) N(g)^{-(s+k)}\right)
\intertext{and}
 Z_-(s) &\sceq Z_-(s,\chi) \sceq \prod_{[g]_{\wt\Gamma} \in [\wt\Gamma]_p} \prod_{k=0}^\infty \det\left( 1 - \det g^{k+1}\cdot \chi(g) N(g)^{-(s+k)}\right),
\end{align*}
respectively. The naming will become clear further below.

For $\Gamma=\Gamma_\ell$ with $\ell=2\cos\tfrac\pi{q}$ with $q\in\N_{\geq 3}$ even, the $\wt\Gamma$-conjugacy classes of \mbox{$\wt\Gamma$-}primi\-tive elements in $\wt\Gamma$ is not bijective to the primitive periodic billiards on $\wt\Gamma_\ell\backslash\h$. In fact, let 
\begin{equation}\label{elemmu}
 g_\mu\sceq \frac{1}{\sin\frac{\pi}{q}} \bmat{1}{\cos\frac\pi{q}}{\cos\frac\pi{q}}{1}.
\end{equation}
Then $g_\mu$ and $Qg_\mu=g_\mu Q$ are both $\wt\Gamma$-primitive but they are not $\wt\Gamma$-conjugate. Their $\wt\Gamma$-conjugacy classes $[g_\mu]_{\wt\Gamma}$ and $[Qg_\mu]_{\wt\Gamma}$ are both associated to the primitive periodic billiard on $\wt\Gamma\backslash\h$ that is represented by the geodesic from $-1$ to $1$ on $\h$. This, however, is the only obstacle towards a bijection. Between all other $\wt\Gamma$-conjugacy classes of primitive elements and all other primitive periodic billiards the standard correspondence is valid. In order to state the definition of even and odd Selberg zeta functions for $(\wt\Gamma,\chi)$ let 
\begin{equation}\label{setmu}
 [\wt\Gamma]_{p,\mu}\sceq \left\{ [g]_{\wt\Gamma} \in  [\wt\Gamma]_p \ \left\vert\  [g]_{\wt\Gamma}\not=[g_\mu]_{\wt\Gamma},\ [g]_{\wt\Gamma}\not=[Qg_\mu]_{\wt\Gamma} \vphantom{[g]_{\wt\Gamma} \in  [\wt\Gamma]_p} \right.\right\}.
\end{equation}
For $\Rea s \gg 1$ we define  
\begin{align*}
 Z_+(s) &\sceq Z_+(s,\chi) 
 \\
 & \sceq Z_{\mu,\id}(s) Z_{\mu,Q}(s)\prod_{[g]_{\wt\Gamma}\in [\wt\Gamma]_{p,\mu}} \prod_{k=0}^\infty \det\left( 1 - \det g^{k}\cdot \chi(g) N(g)^{-(s+k)}\right) 
 \intertext{and}
 Z_-(s) &\sceq Z_-(s,\chi) 
 \\
 & \sceq  Z_{\mu,\id}(s) Z_{\mu,Q}(s)^{-1}\prod_{[g]_{\wt\Gamma}\in [\wt\Gamma]_{p,\mu}} \prod_{k=0}^\infty \det\left( 1 - \det g^{k+1}\cdot \chi(g) N(g)^{-(s+k)}\right), 
\end{align*}
where
\begin{align*}
Z_{\mu,\id}(s) & \sceq  \prod_{k=0}^\infty \left( \det\left( 1-\chi(g_\mu)N(g_\mu)^{-(s+2k)}\right)  \det\left( 1-\chi(g_\mu)N(g_\mu)^{-(s+1+2k)} \right)  \right)^{\frac12}
\intertext{and}
Z_{\mu,Q}(s) & \sceq \prod_{k=0}^\infty \left( \det\left( \left( 1 - \chi(g_\mu)N(g_\mu)^{-(s+2k)}\right)^{\chi(Q)}\right) \right)^{\frac12} 
\\
& \quad\hphantom{\left( 1 - \chi(g_\mu)N(g_\mu)\right)}\times \left(  \det\left( \left( 1 - \chi(g_\mu)N(g_\mu)^{-(s+1+2k)}\right)^{\chi(Q)}\right)\right)^{-\frac12}.
\end{align*}
The matrix-matrix exponential in the latter formula is defined by 
\[
 A^B = \exp\big( (\log A) B\big)
\]
with the obvious choices for the matrices $A$ and $B$.

For each Hecke triangle group $\Gamma$, the relation between the $\Gamma$-conjugacy classes of $\Gamma$-primitive hyperbolic elements in $\Gamma$ and the $\wt\Gamma$-conjugacy classes of $\wt\Gamma$-primitive hyperbolic elements in $\wt\Gamma$ yields that 
\begin{equation}\label{zeta_factorization}
 Z = Z_+\cdot Z_-.
\end{equation}
If $\chi$ is the trivial one-dimensional representation then \eqref{zeta_factorization} is shown in \cite[Theorem~6.2]{Pohl_hecke_infinite} for $\Gamma=\Gamma_\ell$ with $\ell>2$. For $\Gamma=\Gamma_\ell$ with $\ell<2$ it follows from the combination of \cite[Theorem~4.12]{Moeller_Pohl} with \cite[Theorems~5.1 and 6.1]{Pohl_spectral_hecke}. The proof for $\Gamma_2$ is analogous to those for $\ell\not=2$. The generalization to arbitrary finite-dimensional unitary representations $\chi$ can be achieved as in \cite{Pohl_representation}. For the convenience of the reader we provide more details in Section~\ref{sec:SZF_proof} below.

All these Selberg zeta functions admit meromorphic continuations to all of $\C$. For various combinations $(\Gamma,\chi)$ it is known that the spectral parameters for $(\Gamma,\chi)$-cusp forms (and more generally, the resonances) are among the zeros of the Selberg zeta function for $(\Gamma,\chi)$. Even more, for some combinations it is also known that the Selberg zeta functions $Z_{\pm}$ encode the splitting of the spectrum into odd ($-$) and even ($+$) parts \cite{Pohl_spectral_hecke} (see also \cite{Venkov_book}). 

\subsection{Actions}\label{sec:actions}

Let $s\in \C$ and $g\in \Gamma$. For any subset $I$ of $\R$, any function $f\colon I\to V$ and $x\in\R$ such that $g.x\in I$ we define
\begin{equation}\label{alpha_action}
 \alpha_s(g^{-1})f(x) \sceq |g'(x)|^s \chi(g^{-1}) f(g.x)
\end{equation}
whenever it makes sense. We remark that $\alpha_s$, as it is defined here, is not an action of $\Gamma$ on some space of functions. However, for the combinations of functions $f$ and elements $g_1,g_2\in\Gamma$ for which we use \eqref{alpha_action}, the functoriality relation $\alpha_s(g_1g_2)f = \alpha_s(g_1)\alpha_s(g_2)f$ is typically satisfied. Therefore, allowing ourselves a slight abuse of concepts, we refer to $\alpha_s$ as `action'.

In order to define a highly regular (continuous respectively holomorphic) continuation of the action by $\alpha_s$ to all of $\wt\Gamma$ and to functions defined on subsets of $\C$ we define the  action of $g=\textbmat{a}{b}{c}{d} \in \wt\Gamma$ on the Riemann sphere $P^1\C$ by fractional linear transformation:
\begin{equation}\label{action_ext}
 g.z \sceq \frac{az+b}{cz+d}.
\end{equation}
In the case of division by $0$ we identify the fraction with $\infty \in P^1\C$. Note that for $g\in \wt\Gamma$ with $g\notin \Gamma$, the map $g$ in \eqref{action_ext} does not define a Riemannian isometry on $\h$. 

We consider the complex plane $\C$ as embedded into $P^1\C$. Using the identification that $-d/c=\infty$ for $c=0$, \eqref{action_ext} defines a holomorphic map $\C\setminus\{-d/c\}\to\C$ (thus, a holomorphic map $\C\to\C$ if $c=0$).

Further, for $x\in\R\setminus\{-d/c\}$ (i.\,e.\@ for all $x\in\R$ in case that $c=0$) we have
\begin{equation}\label{real}
 |g'(x)|^s = \left( |ad-bc|\cdot (cx+d)^{-2}\right)^s = |ad-bc|^s |cx+d|^{-2s}.
\end{equation}

We use the principal branch for the complex logarithm (i.\,e., with the cut plane $\C\setminus (-\infty,0]$). For the holomorphic continuation of \eqref{real} we then have two possibilities depending on whether we extend the first or the second expression. To that end we choose a representative $\textmat{a}{b}{c}{d}$ of $g$ in $\GL_2(\R)$ such that $c\geq 0$. In case that $c=0$ we choose $d>0$.

From the point of view of transfer operators, the first expression is the more natural one. It extends by 
\[
 j_s^{(1)}(g,z) \sceq \left( |ad-bc|\cdot (cz+d)^{-2}\right)^s
\]
holomorphically to 
\[
 C_{(1)} \sceq \{ z\in \C \mid \Rea z > -d/c \}.
\]
For other approaches to and applications of period functions the second expression is sometimes used. It extends by
\[
 j_s^{(2)}(g,z) \sceq |ad-bc|^s (cz+d)^{-2s}
\]
holomorphically to
\[
 C_{(2)}\sceq \C\setminus(-\infty, -d/c].
\]
Obviously, on $C_{(1)}$ both extensions are identical.  For $k\in \{1,2\}$, any subset $W\subseteq C_{(k)}$, any function $f\colon W \to V$ and $z\in\C$ with $g.z\in W$ and such that $j_s^{(k)}(g,z)$ is defined we set
\begin{equation}\label{defalpha12}
 \alpha_s^{(k)}(g^{-1})f(z) \sceq j_s^{(k)}(g,z)\chi(g^{-1}) f(g.z).
\end{equation}
We write just $\alpha_s$ for generic results or if the choice is understood. The statements and proofs of Theorems~A and B do not depend on this choice. It only affects an intermediate result on the maximal domain of holomorphy for certain functions, see Propositions~\ref{prop:slow_extension} and \ref{prop:fast_extension} below.

\subsection{Meromorphic continuations}\label{sec:meromorphic}

Let $h\in \Gamma$ be a parabolic element. For all $s\in\C$ with $\Rea s >\nobreak \tfrac12$, the infinite sum
\begin{equation}\label{gen_op}
 \mc N_s\sceq\sum_{k=1}^\infty \alpha_s(h^k)
\end{equation}
defines an operator between various spaces of functions, for examples see Sections~\ref{sec:fastTO} and \ref{sec:proof} below or \cite{Moeller_Pohl}. Taking advantage of the Lerch zeta function, either in the form 
\[
 \zeta(s,a,w) = \sum_{n=0}^\infty \frac{e^{2\pi i n a}}{((n+w)^2)^{s/2}}
\]
if we use $\alpha_s^{(1)}$ for $\alpha_s$, or in the form
\[
 \zeta(s,a,w) = \sum_{n=0}^\infty \frac{e^{2\pi i n a}}{(n+w)^s}
\]
if we use $\alpha_s^{(2)}$ for $\alpha_s$, and of its meromorphic continuation one deduces that the map
\[
 s\mapsto \mc N_s
\]
extends meromorphically to all of $\C$. All its poles are simple and contained in $\tfrac12-\tfrac12\N_0$. The existence of poles intimately depends on the degree of singularity of the representation $\chi$ (cf.\@ \cite{Pohl_representation}). 

Throughout, for any operator of the form~\eqref{gen_op}, we denote its meromorphic continuation by $\mc N_s$ as well (more precisely, with the same symbol as the inital operator for $\Rea s > \tfrac12$). Further, to simplify notation, we use $\mc N_s$ to denote any operator which acts by \eqref{gen_op}. The specific spaces on which we consider its action are always understood. Finally, whenever we use an expression that involves $\mc N_s$ and `all' $s\in\C$ then it is understood that we exclude the poles.

\subsection{Transfer operators}

Let $F\colon D\to D$ be a discrete dynamical system. The associated transfer operator $\TO_{\varphi,w}$ with potential $\varphi\colon D\to \C$ and weight function $w$ is defined by
\[
 \TO f(x) \sceq \sum_{y\in F^{-1}(x)} w(y) e^{\varphi(y)} f(y),
\]
acting on an appropriate space of functions $f$ (to be adapted to the discrete dynamical system and the applications under consideration).

The transfer operators we consider in this article have been developed in \cite{Moeller_Pohl, Pohl_hecke_infinite, Pohl_spectral_hecke, Pohl_representation}. We survey their common properties that are important for the proofs of Theorems~A and B. We refer to the original articles as well as to the following sections for more details. 

Let $\Gamma$ denote a Hecke triangle group and let $\wt\Gamma \subseteq \PGL_2(\R)$ be its underlying triangle group. The discrete dynamical systems $(D,F)$ that we use in the transfer operator for $\Gamma$ arise from a discretization and symbolic dynamics for the geodesic flow on $X=\Gamma\backslash\h$ (or rather $\wt\Gamma\backslash\h$). The set $D$ is a family of real intervals $D_\kappa$, $\kappa\in K$ for some (finite or countable) index set $K$, and the map $F$ is determined by a family 
\begin{equation}\label{submaps}
 F_k \sceq F\vert_{D_k} \colon D_k \to F_k(D_k)
\end{equation}
of diffeomorphisms that are identical to the action of certain elements in $\wt\Gamma$. The potentials we are interested in are $\varphi_s(y) = -s\log |F'(y)|$ for $s\in\C$. The weight function depends on the finite-dimensional unitary representation $(V,\chi)$ and whether we intend to investigate the odd (`$-$') or the even (`$+$') spectrum of $\Delta=\Delta(\Gamma,\chi)$.

For the parameter $s\in C$, we denote the even transfer operator by $\TO_s^+$ and the odd transfer operator by $\TO_s^-$. Since we consider the representation $(V,\chi)$ to be fixed throughout, we omit it from the notation. 

For a subset $I\subseteq \R$ let  
\[
\Fct(I;V) \sceq \{ f\colon I\to V\}
\]
denote the space of functions $I\to V$. Formally, any arising transfer operator $\TO_s^\pm$ is represented by a matrix
\[
 \TO_s^\pm = \left( \TO_{s,a,b}^\pm \right)_{a,b\in \mc A}
\]
for a finite index set $\mc A$ and acts on function vectors
\[
 f = \left(f_a\right)_{a\in\mc A}
\]
where, for each $a\in \mc A$,
\[
 f_a \in \Fct(I_a;V)
\]
for some interval $I_a\subseteq\R$. The intervals are closely related to the sets $F_k(D_k)$ in \eqref{submaps}. Further, for any $a,b\in \mc A$ there is a (finite or countable) index set $C_{a,b}$ and for each $c\in C_{a,b}$ an element $g_c^{(a,b)} \in \wt\Gamma$ such that
\begin{equation}\label{matrixelement}
 \TO_{s,a,b}^\pm = \sum_{c\in C_{a,b}} w\big(g_c^{(a,b)}\big)\alpha_s\big( g_c^{(a,b)} \big).
\end{equation}
The weight function is given by $w\colon G\to \{\pm 1\}$, 
\[
w(g) \sceq 
\begin{cases}
1 & \text{for even (`$+$') transfer operators}
  \\
\sign(\det(g)) & \text{for odd (`$-$') transfer operators.}
\end{cases}
\]
Recall that the action $\alpha_s$ depends on the representation $\chi$. Moreover, for any $a,b\in\mc A$ and $c\in C_{a,b}$ we have
\[
 \left( g_c^{(a,b)}\right)^{-1}. I_a \subseteq I_b.
\]
While this latter property ensures well-definedness for each single summand in \eqref{matrixelement}, there might be a convergence problem for the potentially infinite sums.

As indicated in Figure~\ref{fig:program}, the discretizations and symbolic dynamics we use here come in pairs: a slow version and a fast version. The fast version is deduced from the slow one by a certain induction process on certain parabolic elements; we refer to \cite{Moeller_Pohl, Pohl_Symdyn2d, Pohl_hecke_infinite, Pohl_spectral_hecke} for details. Therefore, also the odd and even transfer operators come in pairs: the slow odd and even transfer operators $\TO_s^{\slow,\pm}$ for which all index sets $C_{a,b}$ in \eqref{matrixelement} are finite, and the fast odd and even transfer operators which also have infinite terms.

\subsubsection{Slow transfer operators}\label{sec:asdec}

For the odd and even slow transfer operators $\TO_s^{\slow,\pm}$ for Hecke triangle groups $\Gamma$, the index set $\mc A$ consists of a single element only. For this reason we omit it from the notation. The index set $C$ is finite, its precise number of elements depends on $\Gamma$. Thus, the slow transfer operators indeed act on $\Fct(I;V)$. For our applications we consider them to act on the space $C^\omega(I;V)$ of real-analytic functions,  and we are interested in the space (`real-analytic odd/even \textbf{S}low \textbf{E}igen\textbf{F}unctions for the parameter $s$')
\[
 \SFE^{\omega,\pm}_s \sceq \left\{ f\in C^\omega(I;V) \ \left\vert\  \TO_s^{\slow,\pm}f = f \right.\right\},
\]
more precisely, in a certain subspace $\SFE_s^{\omega,\hol,\pm}$ of functions admitting a \textbf{hol}omor\-ph\-ic extension to a large domain, a certain subspace $\SFE_s^{\omega,\as,\pm}$ of functions satisfying certain growth restrictions (certain \textbf{as}ymptotic behavior) as well as a certain subspace $\SFE_s^{\omega,\dec,\pm}$ of functions obeying certain \textbf{dec}ay properties. These properties depend on the specific Hecke triangle group, for which reason we refer to Sections~\ref{sec:iso_finite}-\ref{sec:iso_nonco} below for the definitions. 

\begin{thm}[\cite{Moeller_Pohl, Pohl_spectral_hecke, Pohl_representation}]\label{slow_props}
Let $\Gamma$ be a cofinite Hecke triangle group, $\chi$ be the trivial character, and $\Rea s \in (0,1)$. Then $\SFE^{\omega,\dec,\pm}_s$ is isomorphic to the space of odd (if `$-$') and even (if `$+$') Maass cusp forms with spectral parameter $s$ for $\Gamma$, respectively.
\end{thm}

If $\Gamma$ is a non-cofinite Hecke triangle groups or $\chi$ not the trivial character or $s\in\C$ lies outside the domain $\{ 0<\Rea s<1\}$ then the spectral interpretation of the sets $\SFE^{\omega,\dec,\pm}_s$ is not yet understood. However, well-supported conjectures exist. We refer to Section~\ref{remarks} for a more detailed discussion.

\subsubsection{Fast transfer operators}\label{sec:fastTO}

For any fast transfer operator, at least one of the index sets $C_{a,b}$ in \eqref{matrixelement} is infinite and hence causes a convergence problem. However, the structure of the infinite sums is controlled and allows for a uniform treatment.

The purpose of the fast transfer operators is to represent Selberg zeta functions as Fredholm determinants. To that end we use a certain Banach space (defined further below) on which the fast transfer operator acts as a nuclear operator of order $0$. This Banach space essentially is the space of function vectors $(f_a)_{a\in\mc A}$ such that each function $f_a$ is real-analytic on $I_a$, extends continuously to $\overline I_a$, extends holomorphically to a complex neighborhood $\mc E_a$ common for all functions $f_a$, and the family of complex neighborhoods $(\mc E_a)_{a\in\mc A}$ is compatible with the mapping properties of the transfer operator. 

To be more precise, let $(\mc E_a)_{a\in\mc A}$ be a family of open connected subsets of the Riemann sphere $\wh\C = \C\cup\{\infty\}$ such that 
\begin{enumerate}[(A)]
\item\label{bigi} for each $a\in\mc A$, the set $\mc E_a$ is a complex neighborhood  (in the Riemann sphere) of the closure $\overline I_a$ of the real interval $I_a$, and
\item\label{bigii} for all $a,b\in \mc A$ and all $c\in C_{a,b}$ we have
\[
 \left( g_c^{(a,b)} \right)^{-1}. \overline{\mc E_a} \subseteq \mc E_b.
\]
\end{enumerate}
Define
\[
 B(\mc E_a) \sceq \{ \text{$\psi\colon \overline{\mc E_a} \to V$ continuous} \mid \text{$\psi\vert_{\mc E_a}$ holomorphic} \}.
\]
Endowed with the supremum norm, $B(\mc E_a)$ is a Banach space. Let
\[
 B(\mc E) \sceq \bigoplus_{a\in \mc A} B(\mc E_a)
\]
be the direct sum of these Banach spaces. As stated in Theorem~\ref{fast_props} below, for $\Rea s>\tfrac12$, each of the fast transfer operators $\TO_s^{\fast,\pm}$ acts on $B(\mc E)$.

The role of the family $(\mc E_a)_{a\in\mc A}$ is to provide a thickening into the complex plane of the family $(\overline I_a)_{a\in\mc A}$  of real intervals and to fix a common domain of holomorphy of the considered function vectors. This thickening is needed in the proof of Theorem~\ref{fast_props} below, in particular for  Grothendieck's theory of nuclear operators on Banach spaces, see \cite{Pohl_hecke_infinite, Pohl_spectral_hecke, Pohl_representation}. However, none of the results in this paper depends on the specific choice of the family $(\mc E_a)_{a\in\mc A}$. Thus it would be natural to consider the inductive limit of the operators $\TO_s^{\fast,\pm}\colon B(\mc E)\to B(\mc E)$, where the system is directed by shrinking domains (thus, if $\mc E = (\mc E_a)_{a\in\mc A}$ and $\mc E'=(\mc E'_a)_{a\in\mc A}$ then 
\[
 \left( \TO_s^{\fast,\pm}\colon B(\mc E) \to B(\mc E) \right) \ \preccurlyeq\  \left( \TO_s^{\fast,\pm}\colon B(\mc E') \to B(\mc E')  \right)
\]
if and only if $\mc E'_a \subseteq \mc E_a$ for all $a\in \mc A$). We omit here a further discussion of this limit and its topological properties, and  
work with a fixed family $\mc E = (\mc E_a)_{a\in\mc A}$. To emphasize the independence of all results from this choice we use also the suggestive notation 
\[
 \mc B \sceq \mc B(I) \sceq B(\mc E)
\]
in order to stress that the family $I=(I_a)_{a\in\mc A}$ is the essential structure and the family of complex neighborhoods a rather auxiliary object.

\begin{thm}[\cite{Pohl_hecke_infinite, Pohl_spectral_hecke, Pohl_representation}] \label{fast_props}
\begin{enumerate}[{\rm (i)}]
\item For $\Rea s > \tfrac12$, each transfer operator $\TO_s^{\fast,\pm}$ acts on $\mc B$ as a nuclear operator of order $0$. 
\item The map $s\mapsto \TO_s^{\fast,\pm}$ extends to a meromorphic function on $\C$ with values in nuclear operators of order $0$ on $\mc B$. The possible poles are all simple and contained in $\tfrac12(1-\N_0)$.
\item The Selberg zeta function $Z$ for $(\Gamma,\chi)$ equals the Fredholm determinant
\[
 Z(s) = \det\left( 1 - \TO_s^{\fast,+}\right)\det\left(1 -\TO_s^{\fast,-}\right).
\]
\item If $\Gamma$ is a lattice with a single cusp and $\chi$ is the trivial one-dimensional representation then $\det(1-\TO_s^{\fast,\pm})$ equals the Selberg-type zeta function $Z_\pm$ for the odd (if `$-$') and the even (if `$+$') spectrum, respectively:
\[
 Z_\pm(s) = \det\left(1-\TO_s^{\fast,\pm}\right).
\]
\end{enumerate}
\end{thm}

For $s\in\C$ we define (`odd/even \textbf{F}ast \textbf{E}igen\textbf{F}unctions for the parameter $s$')
\[
 \FFE_s^\pm \sceq \left\{ f\in \mc B \ \left\vert\ f=\TO_s^{\fast,\pm}f \right.\right\}.
\]
The elements of $\FFE_s^\pm$ determine the zeros of $Z_\pm$, respectively, and hence by \eqref{zeta_factorization} those of $Z$.

\subsection{Notation} For any $x_0\in \R\cup\{\pm \infty\}$ and any functions $f,g\colon \R\to\C$ we use $f(x) = O_{x\to x_0^+}(g(x))$ for   
\[
\limsup_{x\searrow x_0} \left|\frac{f(x)}{g(x)}\right| < \infty.
\]
Note that, in contrast to other conventions, we allow (for simplicity) that $g$ does not need to be positive. We use analogous conventions for the other symbols from the $O$-notation. 

Further, for functions $f,g\colon D\to\C$ with $D\subseteq \C$ we use $f\ll g$ if there exists $C>0$ such that for all $x\in D$ we have
\[
 |f(x)| \leq C |g(x)|.
\]
Moreover, we say that $f$ satisfies a property $P$ for all $|x|\gg 1$ if there exists $C\geq 0$ such that for all $|x|\geq C$, $f(x)$ satisfies $P$.

\section{Proof of Theorems~A and B}\label{sec:proof}

We show Theorem~B separately for the cofinite Hecke triangle groups with a single cusp, the Theta group, and the non-cofinite Hecke triangle groups. Within these classes, the structure of the groups and transfer operators allows for an easy uniform statement of the maps that provide the claimed isomorphism between the eigenspaces of the slow and fast transfer operators. 

Recall that 
\[
Q= \bmat{0}{1}{1}{0}\quad\text{and}\quad J= \bmat{-1}{0}{0}{1}.
\]

\subsection{Isomorphism for the Hecke triangle groups $\Gamma_\ell$ with 
$\ell=2\cos(\pi/q)$, $q\geq 3$ odd}\label{sec:iso_finite}

Let $q\in\N_{\geq 3}$ and set
\[
 \ell\sceq\ell(q) \sceq 2\cos\frac{\pi}{q}.
\]
For the cofinite Hecke triangle group 
\[
 \Gamma \sceq \Gamma_q \sceq \Gamma_\ell
\]
with a single cusp we consider the transfer operators developed in 
\cite{Moeller_Pohl, Pohl_spectral_hecke, Pohl_representation}. We recall their definitions and major properties.

To that end recall that $S = 
\textbmat{0}{1}{-1}{0}$ and $T \sceq T_q \sceq T_\ell = 
\textbmat{1}{\ell}{0}{1}$. For $k\in \Z$ let
\begin{equation}\label{defgqk}
 g_{q,k} \sceq \left( \big(T_qS)^k S\right)^{-1},
\end{equation}
and, for $m\in\Z$, set
\[
 \s(m,q) \sceq \frac{\sin\left(\frac{m}{q}\pi\right)}{\sin\frac{\pi}{q}}.
\]
Then we have
\[
 g_{q,k}^{-1} = \bmat{\s(k,q)}{\s(k+1,q)}{\s(k-1,q)}{\s(k,q)}.
\]
Thus, $g_{q,k}^{-1}$ is $q$-periodic in the variable $k$. The elements 
\begin{equation}\label{parab_explicit}
 g_{q,1}^{-1} = g_{q,nq-1}^{-1} = \bmat{1}{\ell}{0}{1},\quad g_{q,-1}^{-1} = g_{q,nq+1}^{-1} = \bmat{1}{0}{\ell}{1}\qquad (n\in\Z)
\end{equation}
are parabolic, the elements 
\[
g_{q,0}^{-1} = g_{q,nq}^{-1} = \id \qquad (n\in\Z)
\]
are the identity element (and will not play any role in the following). All the remaining elements are hyperbolic.

Let 
\[
 m\sceq \left\lfloor \frac{q-1}{2} \right\rfloor.
\]

In this section we consider the case of $q$ odd. The case for even $q$ is essentially identical (treated in Section~\ref{sec:iso_even} below), the only difference is the explicit formula for the transfer operators. Thus, let $q$ be odd. Then
\[
 m=\frac{q-1}{2}.
\]
To simplify notation, we omit throughout the subscripts $q$ and $\ell$.

\subsubsection{Slow transfer operators for odd $q$} The odd (`$-$') and even (`$+$') slow transfer operator $\TO^{\slow, \pm}_{s}$ of 
$\Gamma$ is given by
\begin{align*}
\TO^{\slow, \pm}_{s} & =  \sum_{k=1}^{m} \alpha_s(g_{-k}) \pm 
\alpha_s(Qg_{-k}) 
 \\
& =  \big( 1 \pm \alpha_s(Q)\big)  \sum_{k=1}^{m} \alpha_s(g_{-k}),
\end{align*}
acting on $C^\omega( (0,1);V)$. Let
\[
\SFE_{s}^{\omega,\pm}\sceq \left\{ \varphi\in C^\omega\big( (0,1); V\big) \ 
\left\vert\ \varphi = \TO_{s}^{\slow,\pm}\varphi\right.\right\}
\]
denote the space of real-analytic bounded eigenfunctions of 
$\TO_{s}^{\slow,\pm}$, respectively, with eigenvalue $1$.

For $z\in\C$, $r>0$ let 
\[
 B_r(z) \sceq \{ w\in \C \mid |w-z|<r\}
\]
denote the open ball in $\C$ around $z$ with radius $r$. Let $I=(a,b)\subseteq\R$ be a finite interval, $\mc U$ be a complex neighborhood of $I$ and $\eps>0$. We say that $\mc U$ is an \textit{$\eps$-rounded neighborhood} if there exists $\eps>0$ such that 
\[
 B_\eps(a+\eps)\cup B_\eps(b-\eps)\cup \big( (a+\eps,b-\eps)+i(-\eps,\eps) \big) \subseteq\mc U.
\]
We call $\mc U$ \textit{$\eps$-rounded at $a$} if 
\[
 B_\eps(a+\eps)\subseteq\mc U,
\]
and, analogously, that $\mc U$ is \textit{$\eps$-rounded at $b$} if
\[
 B_\eps(b-\eps)\subseteq\mc U.
\]
We say $\mc U$ is a \textit{rounded neighborhood} if there exists $\eps>0$ such that $\mc U$ is $\eps$-rounded. Analogously we define the notions \textit{rounded at} $a$ or $b$.

Let $\SFE_{s}^{\omega,\hol,\pm}$ denote the space of functions $\varphi\in\SFE_{s}^{\omega,\pm}$ for which there exists a complex neighborhood $\mc U = \mc U(\varphi)$ of $(0,1)$ that is rounded at $0$ and to which $\varphi$ has a holomorphic extension $\wt\varphi$ that satisfies the functional equation
\begin{equation}\label{slow_fe}
\wt\varphi =\sum_{k=1}^m \big( \alpha_s(g_{-k}) \pm \alpha_s(Qg_{-k}) \big)\wt\varphi.
\end{equation}

Let
\begin{equation}\label{def_sfe}
 \SFE_{s}^{\omega,\as,\pm}\sceq \left\{ \varphi\in \SFE_{s}^{\omega,\hol,\pm} \ 
\left\vert\ \exists\, c\in V\colon \varphi(x) = \frac{c}{x} + O_{x\to0^+}(1) 
\right.\right\}
\end{equation}
denote its subspace of functions with a certain controlled growth towards $0$, 
and let $\SFE_{s}^{\omega,\dec,\pm}$ denote its subspace of functions 
$\varphi\in \SFE_{s}^{\omega,\pm}$ for which the map
\begin{equation}\label{phi_decay}
\begin{cases}
\varphi & \text{on $\left(0,\tfrac1{\ell}\right)$}
\\[2mm]
\mp\alpha_s(J)\varphi & \text{on $\left(-\tfrac1{\ell},0\right)$}
\end{cases}
\end{equation}
extends smoothly ($C^\infty$) to $\big(-1/\ell,1/\ell\big)$. As already indicated in Section~\ref{sec:asdec}, the superscript `$\as$' abbreviates `asymptotic behavior', refering to the growth towards $0$. The superscript `$\dec$' abbreviates `decay', refering to the necessary decay behavior of $\varphi$ in order to satisfy \eqref{phi_decay}.

\begin{lemma}\label{lem:nhbs}
Let $\varphi\in \SFE_{s}^{\omega,\hol,\pm}$. Then $\varphi$ extends holomorphically to a rounded neighborhood $\mc W$ of $(0,1)$ and its extension satisfies \eqref{slow_fe} on all of $\mc W$.
\end{lemma}

\begin{proof}
By hypothesis, we find a complex neighborhood $\mc U$ of $(0,1)$ that is rounded at $0$, to which $\varphi$ extends holomorphically and on which this extension, also denoted by $\varphi$, satisfies \eqref{slow_fe}. Since $g_{-1}^{-1}$ is parabolic with fixed point $0$, $g_{-2}^{-1},\ldots, g_{-m}^{-1}$ are hyperbolic with attracting fixed points contained in the interval $(0,1)$, all repelling fixed points are bounded away from $(0,1)$, and $Q$ fixes $1$, we find a complex neighborhood $\mc V$ of $1$ such that $Q.\mc V\subseteq \mc V$ and
\[
 g_{-k}^{-1}.\mc V \subseteq \mc U\quad \text{for $k=1,\ldots, m$.}
\]
Then 
\[
 \psi\sceq \sum_{k=1}^m \big( \alpha_s(g_{-k}) \pm \alpha_s(Qg_{-k}) \big)\varphi
\]
defines a holomorphic function on $\mc W\sceq \mc U\cup\mc V$. Further, $\psi$ coincides with $\varphi$ on $\mc U$ since $\varphi$ satisfies \eqref{slow_fe} on all of $\mc U$. By the identity theorem of holomorphic functions, $\psi$ satisfies \eqref{slow_fe} on all of $\mc W$. Obviously, $\mc W$ is a rounded neighborhood of $(0,1)$.
\end{proof}

\begin{remark}\label{rem:stronger}
In Corollary~\ref{cor:vanish} below we will see that the elements of $\SFE_s^{\omega,\as,\pm}$ satisfy  stronger asymptotics than requested in \eqref{def_sfe} towards the cusp of $X_\ell$ in all directions that are `closed' by the representation $\chi$. To be more precise let
\[
 E_1\sceq \{ v\in V \mid \chi(g_{-1})v=v \},
\]
let $E_r$ be the orthogonal complement of $E_1$ in $V$, and let
\[
 \pr_r\colon V\to E_r
\]
be the orthogonal projection on $E_r$. Then every $\varphi\in \SFE_s^{\omega,\as,\pm}$ satisfies
\[
 \varphi(x) = \frac{c}{x} + O_{x\to 0^+}(1)
\]
for some $c\in V$ with $\pr_r(c) = 0$, at least if $s\in\C$, $\Rea s > 0$, $s\not=1/2$.

The property $\pr_r(c)=0$ means that in all directions of the cusp that are not stabilized by $\chi$, the function $\varphi$ behaves as if the space is closed. 
\end{remark}

\begin{remark}\label{rem:bdd}
For each $\varphi\in \SFE_{s}^{\omega,\dec,+}$ the condition~\eqref{phi_decay} 
implies that we have 
\[
\lim_{x\to0^+}\varphi(x) = 0.
\]
Even more, since the limit $\lim_{x\to0^+}\varphi'(x)$ exists, 
\[
 \varphi = O_{x\to0^+}(x).
\]
\end{remark}

\begin{remark}\label{otherdomain}
In \cite{Moeller_Pohl, Pohl_spectral_hecke} (isomorphism between Maass cusp forms and eigenfunctions of transfer operators) we consider $\TO_{s}^{\slow,\pm}$ 
to act on $C^\omega( \R_{>0};V)$ instead of on $C^\omega( (0,1);V)$ and require 
that
\begin{equation}\label{phi_decay2}
\begin{cases}
\varphi & \text{on $\R_{>0}$}
\\
-\alpha_s(S)\varphi & \text{on $\R_{<0}$}
\end{cases}
\end{equation}
extends smoothly to $\R$ instead of asking for \eqref{phi_decay}. However, if 
$\varphi\in C^\omega(\R_{>0};V)$ is an eigenfunction with eigenvalue $1$ of 
$\TO_{s}^{\slow,\pm}$ then $\varphi = \pm\alpha_s(Q)\varphi$. Substituting 
this into \eqref{phi_decay2} and noting that $SQ=J$ shows that 
\eqref{phi_decay2} is equivalent to \eqref{phi_decay} up to real-analyticity at 
$1$. However, Proposition~\ref{prop:slow_extension} below shows that each 
element of $\SFE_{s}^{\omega,\pm}$ extends uniquely to an element in 
$C^\omega(\R_{>0};V)$. Thus, \eqref{phi_decay} and \eqref{phi_decay2} are indeed 
equivalent.
\end{remark}

\subsubsection{Fast transfer operators for odd $q$} In order to state the fast odd (`$-$') and even (`$+$') transfer operator 
$\TO^{\fast,\pm}_{s}$ of $\Gamma$ we set 
\begin{equation}\label{def_sets}
D_{-1} \sceq \left(0,\tfrac1{\ell}\right)\quad\text{and}\quad D_{0}\sceq 
\left(\tfrac1{\ell},1\right)
\end{equation}
as well as
\[
 \TO_{0,s}^\fast  \sceq  \sum_{k=2}^{m} \alpha_s( g_{-k} ).
\]
For $\Rea s > \tfrac12$ we set
\begin{equation}\label{TOm1_def}
 \TO_{-1,s}^\fast \sceq \sum_{n=1}^\infty \alpha_s( g_{-1}^n ),
\end{equation}
and have
\[
\TO_{s}^{\fast,\pm} = 
\begin{pmatrix}
\big(1\pm \alpha_s(Q)\big) \TO_{0,s}^{\fast} & \big(1\pm \alpha_s(Q)\big) 
\TO_{-1,s}^{\fast}
\\[2mm]
\big(1\pm \alpha_s(Q)\big) \TO_{0,s}^{\fast} & \pm \alpha_s(Q) 
\TO_{-1,s}^{\fast}
\end{pmatrix}
\]
which acts on the Banach space
\[
 \mc B \sceq \mc B(D_{0}) \oplus \mc B(D_{-1}).
\]
For $\Rea s \leq \tfrac12$, $\TO_{-1,s}^\fast$ and $\TO_{s}^{\fast,\pm}$ are 
given by meromorphic continuation (see Theorem~\ref{fast_props} or 
\cite{Moeller_Pohl, Pohl_representation}).

The choice of notation in \eqref{def_sets} refers to the fact that $g_{-1}^{-1}$ maps to $D_{-1}$, and all the other elements $g_{-2}^{-1},\ldots,g_{-m}^{-1}$ map to $D_0$.

For $s\in \C$ let 
\[
 \FFE_{s}^\pm \sceq \left\{ f \in \mc B \ \left\vert\ f = 
\TO_{s}^{\fast,\pm} f \right.\right\}
\]
denote the space of eigenfunctions in $\mc B$ of $\TO_{s}^{\fast,\pm}$ with 
eigenvalue $1$. Let 
$\FFE_{s}^{\dec,\pm}$ denote the subspace of maps $f=(f_{0}, f_{-1})^\top \in 
\FFE_{s}^\pm$ for which the map
\begin{equation}\label{f_decay}
\begin{cases}
\left(1+\TO_{-1,s}^{\fast}\right)f_{-1} & \text{for $x>0$}
\\[1mm]
\mp \alpha_s(J)\left(1+\TO_{-1,s}^{\fast}\right)f_{-1} & \text{for $x<0$}
\end{cases}
\end{equation}
extends smoothly to $0$ when considered as a function on some punctured 
neighborhood of $0$ in $\R$.

\subsubsection{Special case $q=3$}

For $q=3$, i.\,e., for the modular group $\PSL_2(\Z)$, the set $D_0$ 
is empty and hence there is no $f_0$-component. The transfer operators simplify to 
\[
 \TO^{\slow,\pm}_{3,s} = \big(1\pm\alpha_s(Q)\big)\circ \alpha_s(g_{3,-1})
\]
and
\[
 \TO^{\fast,\pm}_{3,s} = \pm\alpha_s(Q)\TO^\fast_{3,-1,s},
\]
which, for $\Rea s > 1/2$, is
\begin{equation}\label{baby3}
 \TO^{\fast,\pm}_{3,s} = \pm\alpha_s(Q)\sum_{n=1}^\infty \alpha_s(g_{3,-1}^n).
\end{equation}
For the case that $\chi$ is the trivial character, \eqref{baby3} coincides with $\pm\TO_s^{\text{\rm Mayer}}$, see \eqref{TOMayer}. It is does not coincide with any other transfer operator existing for $\PSL_2(\Z)$. For a more detailed discussion we refer to \cite[Remark~4.3]{Pohl_mcf_general}.

For $\chi$ being the trivial character, \cite{Lewis_Zagier} and \cite{Chang_Mayer_transop} showed that the map
\begin{equation}\label{iso_modular}
 f_{-1} = \alpha_s(g_{3,1})\varphi,\quad \varphi=\alpha_s(g_{3,1}^{-1})f_{-1}
\end{equation}
provides an isomorphism between the eigenfunctions of $\TO^{\slow,\pm}_{3,s}$ and $\TO^{\fast,\pm}_{3,s}$. To be more precise, at the time of their results, the slow transfer operator had not yet been discovered. They showed an isomorphism between the eigenfunctions with eigenvalue $1$ of $\TO^{\fast,\pm}_{3,s}$ and the solutions (of appropriate regularity) of the functional equation
\[
 \varphi(x) = \varphi(x+1) + (x+1)^{-2s} \varphi\left( \frac{x}{x+1}\right), \quad x\in\R_{>0}
\]
that are invariant (`$+$') respectively anti-invariant (`$-$') under the action of $Q$. In our terms these functions are eigenfunctions with eigenvalue $1$ of $\TO^{\slow,\pm}_{3,s}$. 

The combination of \cite{Chang_Mayer_eigen, Chang_Mayer_extension, Hilgert_Mayer_Movasati, Deitmar_Hilgert, Fraczek_Mayer_Muehlenbruch} shows that \eqref{iso_modular} provides also an isomorphism for certain representations $\chi$. These studies take advantage of the special structure of $\TO^{\fast,\pm}_{3,s}$ which is not present anymore for $q>3$. Therefore, in the general case, the isomorphism, as stated in Theorem~\ref{thm:main_finite} below, is more involved. For the case of $q=3$, one easily sees that the isomorphism in Theorem~\ref{thm:main_finite} reduces to \eqref{iso_modular}.

\subsubsection{Statement of main theorem for odd $q$}

We start with an \textit{informal} abstract deduction of the isomorphism. Every object or step in the following which requires technical justification (e.\,g., raises convergence questions) is dealt with in the actual proof of Theorem~\ref{thm:main_finite} below, see Sections~\ref{sec:holomorphy}-\ref{sec:proofmain} below.

The principal objects for the isomorphism are the \textit{slow} discretizations for the geodesic flow and the \textit{slow} transfer operators. The \textit{fast} discretizations and the \textit{fast} transfer operators arise as follows: Whenever the acting element in the slow discretization is parabolic, one induces on this element in order to construct the fast discretization. More precisely, suppose that $p\in\PSL_2(\R)$ is parabolic with fixed point $a\in\R\cup\{\infty\}$ and suppose further that the slow discrete dynamical system contains a component (submap) of the form
\begin{equation}\label{slow_sm}
 (p^{-1}.b, a) \to (b,a),\quad x\mapsto p.x
\end{equation}
(or $(a, p^{-1}.b)\to (a,b)$, $x\mapsto p.x$). Then, for the fast discretization, this submap is substituted by the maps ($n\in\N$)
\begin{equation}\label{fast_sm}
 (p^{-n}.b, p^{-(n+1)}.b) \to (b, p^{-1}.b),\quad x\mapsto p^n.x.
\end{equation}
Let $1_W$ denote the characteristic function of any set $W$. The map in \eqref{slow_sm} contributes to the slow transfer operator the term
\begin{equation}\label{slow_to}
 1_{(b,a)}\cdot \alpha_s(p),
\end{equation}
the map in \eqref{fast_sm} contributes to the fast transfer operator the term
\begin{equation}\label{fast_to}
 1_{(b,p^{-1}.b)}\cdot \sum_{n\in\N} \alpha_s(p^n).
\end{equation}
We refer to \cite{Moeller_Pohl, Pohl_hecke_infinite, Pohl_spectral_hecke, Pohl_representation} for a detailed description of the induction process and explicit examples. 

In the previous sections we have only provided the (equivalent) matrix representations for transfer operators. We refer to \cite{Moeller_Pohl} for a detailed explanation how to switch between those and \eqref{slow_to}-\eqref{fast_to}. 

At those places where the acting element is hyperbolic, the slow and the fast discretizations are identical. The guiding idea for the isomorphism map is that we want to assign to an eigenfunction $\varphi$ of the slow transfer operator the (unique) eigenfunction $f$ of the fast transfer operator that is `dynamically as identical as possible to $\varphi$', and vice versa. We elaborate this idea to make it more precise.

First let $\varphi$ be a (given) eigenfunction with eigenvalue $1$ of $\TO_s^\slow$. In the following we construct a (unique) candidate for an eigenfunction $f$ with eigenvalue $1$ of $\TO_s^\fast$. At those intervals where the slow and fast discretizations are identical the maps $f$ and $\varphi$ should coincide. Thus, if $I_0$ is an interval arising in a submap (as an image of a map as in \eqref{submaps}!) and the acting element is hyperbolic then we define
\begin{equation}\label{defbasic}
 f\vert_{I_0} \sceq \varphi\vert_{I_0}.
\end{equation}
If $I_p$ is an interval in a submap (again, as an image interval!) where a parabolic element is acting, say $p$ is the parabolic element and $I_p=(b,p^{-1}.b)$ the interval, then on $I_p$, the function $f$ heuristically needs to be the difference between $\varphi$ and one $p$-iterate of $\varphi$. Thus we define 
\begin{equation}\label{defparabdyn}
 f\vert_{I_p} \sceq \big(1-\alpha_s(p)\big)\varphi\vert_{I_p}.
\end{equation}

Now let $f$ be a (given) eigenfunction with eigenvalue $1$ of $\TO_s^\fast$. We want to define a (unique) candidate for an eigenfunction $\varphi$ with eigenvalue $1$ if $\TO_s^\slow$ such that the definitions \eqref{defbasic} and \eqref{defparabdyn} are inverted. Thus, on an interval $I_0$ as for \eqref{defbasic} we set
\[
 \varphi\vert_{I_0} \sceq f\vert_{I_0}.
\]
Suppose that the parabolic element $p$ and the interval $I_p$ are as for \eqref{defparabdyn}. The formal inverse of $(1-\alpha_s(p))$ is
\[
 \sum_{n=0}^\infty \alpha_s(p^n) = 1 + \sum_{n\in\N} \alpha_s(p^n).
\]
Therefore we set
\begin{equation}\label{definverse}
 \varphi\vert_{I_p} \sceq \Bigg(1 + \sum_{n\in\N} \alpha_s(p^n)\Bigg)f\vert_{I_p}.
\end{equation}
A major point in the proof of Theorem~\ref{thm:main_finite} below is to discuss the convergence issues raised by \eqref{definverse} and to establish that it is indeed inverse to \eqref{defparabdyn}. We note already here that as soon as $\varphi=o(x^{-2s})$ is established or assumed, one easily sees that \eqref{definverse} is indeed the inverse to \eqref{defparabdyn}.

\begin{thm}\label{thm:main_finite}
Let $s\in\C\setminus\{\tfrac12\}$ such that $\Rea s > 0$. Then the spaces 
$\SFE_{s}^{\omega,\as,\pm}$ and $\FFE_{s}^\pm$ are isomorphic (as vector 
spaces). The isomorphism is given by
\[
 \FFE_{s}^\pm \to \SFE_{s}^{\omega,\as,\pm},\quad f=(f_{0}, f_{-1})^\top 
\mapsto \varphi,
\]
where
\begin{equation}\label{ftophi}
 \varphi\vert_{D_{0}} \sceq f_{0}\vert_{D_0} \quad\text{and}\quad \varphi\vert_{D_{-1}} 
\sceq \left( 1 + \TO_{-1,s}^\fast\right)f_{-1}\vert_{D_{-1}}.
\end{equation}
The inverse isomorphism is
\[
 \SFE_{s}^{\omega,\as,\pm} \to \FFE_{s}^\pm,\quad \varphi\mapsto f=(f_{0}, 
f_{-1})^\top,
\]
where $f$ is determined by
\begin{equation}\label{phitof}
 f_{0}\vert_{D_0} \sceq \varphi\vert_{D_{0}} \quad\text{and}\quad f_{-1}\sceq 
\big(1-\alpha_s(g_{q,-1})\big)\varphi\vert_{D_{-1}}.
\end{equation}
These isomorphisms induce isomorphisms between $\SFE_{s}^{\omega,\dec,\pm}$ 
and $\FFE_{s}^{\dec,\pm}$.
\end{thm}

If one ignores all questions of convergence and in particular uses \eqref{TOm1_def} for $\TO^\fast_{-1,s}$ then a straightforward formal calculation (converting the heuristics from above) shows that \eqref{ftophi} and \eqref{phitof} indeed map eigenfunctions with eigenvalue $1$ of $\TO^{\fast,\pm}_{s}$ to eigenfunctions with eigenvalue $1$ of $\TO^{\slow,\pm}_{s}$, and vice versa. 

For a rigorous proof of Theorem~\ref{thm:main_finite} we first show two intermediate results. The first one, proven in Section~\ref{sec:holomorphy} below, discusses the maximal domains of holomorphy for the elements of $\SFE_{s}^{\omega,\hol,\pm}$ and $\FFE_{s}^{\pm}$. \textit{A priori}, these elements are defined on different domains: the functions in $\SFE_{s}^{\omega,\hol,\pm}$ are defined on some interval in $\R$ whereas function vectors in $\FFE_{s}^{\pm}$ are defined on certain open sets in $\C$. The result on the maximal domains simplifies to compare the functions in these two spaces.

As a second intermediate result we show, in Section~\ref{sec:crucial} below, that
\[
 \TO_{-1,s}^\fast f_{-1} = \alpha_s(g_{-1})\varphi 
\]
whenever $f=(f_0,f_{-1})^\top \in \FFE_s^{\pm}$ is given and $\varphi$ is 
defined by \eqref{ftophi}, or $\varphi\in \SFE^{\omega,\as,\pm}_s$ is given and 
$f$ is defined by \eqref{phitof}. This is a crucial identity needed for establishing Theorem~\ref{thm:main_finite}.

\subsubsection{Maximal domains of holomorphy}\label{sec:holomorphy}

In order to study the maximal domains of holomorphy for the elements of 
$\SFE_{s}^{\omega,\hol,\pm}$ and $\FFE_{s}^{\pm}$ we start by investigating the contraction properties of the group elements acting in the iterates of the transfer operators.

Let 
\[
 A\sceq \left\{ g_{\pm 1}^{-1},\ldots, g_{\pm m}^{-1} \right\}
\]
be the elements acting  in the transfer operators (the `alphabet'). For each 
$n\in\N_0$, let
\[
 A^n \sceq \left\{ g_{k_1}^{-1}\cdots g_{k_n}^{-1} \ \left\vert\  
\text{$g_{k_j}^{-1} \in A$ for $j=1,\ldots, n$}\right.\right\}
\]
denote the \textit{words} of \textit{length} $n$ over $A$, and let
\[
 A^\ast \sceq \bigcup_{n\in\N_0} A^n
\]
denote the set  of all words over $A$. Further let 
\begin{align*}
 A^n_{-1} &\sceq \left\{ g_{k_1}^{-1}\cdots g_{k_n}^{-1} \in A^n \ \left\vert\  
k_1 = -1 \vphantom{g_{k_1}^{-1}\cdots g_{k_n}^{-1}} \right.\right\},
 \\
 A^n_{(-1,1)} & \sceq \left\{ g_{k_1}^{-1}\cdots g_{k_n}^{-1} \in A^n_{-1} \ 
\left\vert\  k_n = 1 \vphantom{g_{k_1}^{-1}\cdots g_{k_n}^{-1}} \right.\right\},
 \\
 A^n_{(-1,-1)} & \sceq \left\{ g_{k_1}^{-1}\cdots g_{k_n}^{-1} \in A^n_{-1} \ 
\left\vert\  k_n = -1 \vphantom{g_{k_1}^{-1}\cdots g_{k_n}^{-1}} 
\right.\right\},
\end{align*}
and
\begin{align*}
 A^n_{0} &\sceq \left\{ g_{k_1}^{-1}\cdots g_{k_n}^{-1} \in A^n \ \left\vert\ 
k_1 \in \{ -2,\ldots, -m\} \vphantom{g_{k_1}^{-1}\cdots g_{k_n}^{-1}} 
\right.\right\},
 \\
 A^n_{(0,1)} &\sceq \left\{ g_{k_1}^{-1}\cdots g_{k_n}^{-1} \in A^n_0 \ 
\left\vert\ k_n  = 1 \vphantom{g_{k_1}^{-1}\cdots g_{k_n}^{-1}} \right.\right\},
 \\
 A^n_{(0,-1)} &\sceq \left\{ g_{k_1}^{-1}\cdots g_{k_n}^{-1} \in A^n_0 \ 
\left\vert\ k_n  = -1 \vphantom{g_{k_1}^{-1}\cdots g_{k_n}^{-1}} 
\right.\right\},
\end{align*}
as well as
\begin{align*}
 A^\ast_{-1} \sceq \bigcup_{n\in\N_0} A^n_{-1},\quad 
 A^\ast_{(-1,1)} \sceq \bigcup_{n\in\N_0} A^n_{(-1,1)},\quad
 A^\ast_{(-1,-1)} \sceq \bigcup_{n\in\N_0} A^n_{(-1,-1)}
\end{align*}
and
\begin{align*}
 A^\ast_{0} \sceq \bigcup_{n\in\N_0} A^n_{0},\quad 
 A^\ast_{(0,1)} \sceq \bigcup_{n\in\N_0} A^n_{(0,1)},\quad
 A^\ast_{(0,-1)} \sceq \bigcup_{n\in\N_0} A^n_{(0,-1)}
\end{align*}
Let 
\[
 \C_R \sceq \{z\in\C\mid \Rea z \geq 0\}.
\]
We recall the sets $D_{-1}=(0,1/\ell)$ and $D_0=(1/\ell,1)$ from \eqref{def_sets}. Throughout and in particular in the following lemma, the notion of finite sets includes the empty set.

\begin{lemma}\label{lem:allcontracts}
Let $\mc U_{-1}$ be a rounded neighborhood of $D_{-1}$, and $\mc U_{0}$ a 
rounded neighborhood of $D_{0}$. Let $\mc U\subseteq \C$ be an open bounded set 
that is bounded away from $(-\infty,0]$, and let $\mc V\subseteq \C$ be an open 
bounded set that is bounded away from $(-\infty,-1/\ell]$. Then the following 
properties are satisfied:
\begin{enumerate}[{\rm (i)}]
\item\label{lem:alli} For all but finitely many $g\in A^\ast_{-1}$ we have
$g.\mc U \subseteq \mc U_{-1}$ and $gQ.\mc U \subseteq \mc U_{-1}$ and 
$g.(\mc U_{-1}\cap \C_R) \subseteq \mc U_{-1}\cap \C_R$. 
\item For all but finitely many $g\in A^\ast_{0}$ we have $g.\mc U 
\subseteq \mc U_{0}$ and $gQ.\mc U \subseteq \mc U_{0}$ and 
$g.(\mc U_{0}\cap \C_R) \subseteq \mc U_{0}\cap \C_R$. 
\item For all but finitely many $g\in A^\ast_{0}\setminus A^\ast_{(0,-1)}$ 
we have $g.\mc V \subseteq \mc U_{0}$.
\item For all but finitely many $g\in A^\ast_{0}\setminus A^\ast_{(0,1)}$ we have $gQ.\mc V \subseteq \mc U_{0}$.
\item For all but finitely many $g\in A^\ast_{-1}\setminus A^\ast_{(-1,-1)}$ 
we have $g.\mc V \subseteq \mc U_{-1}$.
\item For all but finitely many $g\in A^\ast_{-1}\setminus A^\ast_{(-1,1)}$ 
we have  $gQ.\mc V \subseteq \mc U_{-1}$.
\end{enumerate}
\end{lemma}

We recall from \eqref{parab_explicit} (and the text below it) that the elements $g_{\pm 1}^{-1}$ are parabolic with fixed point $0$ and $\infty$, respectively, and that all the elements $g_{\pm 2}^{-1},\ldots, g_{\pm m}^{-1}$ are hyperbolic with attracting fixed points in $(0,\infty)$ (bounded away from $0$ and $\infty$). Lemma~\ref{lem:allcontracts} follows from the contraction properties of the action of combinations of these group elements. Its proof can essentially be read off from 
Figures~\ref{fig:forward} and \ref{fig:backward}. Before we provide a rather detailed proof further below, we sketch how these two figures indicate the proof of Lemma~\ref{lem:allcontracts}\eqref{lem:alli}.

Figure~\ref{fig:forward} indicates 
the location of $g.\C_R$ for $g\in A^\ast$. It shows that if $\mc W$ is a subset of  
$\C_R$ then $h.\mc W \subseteq \mc U_{-1}$ for all sufficiently 
long words $h\in A^\ast_{-1}$. Since $\C_R$ is invariant under the action of 
$Q$, it also follows that $hQ.\mc W\subseteq \mc U_{-1}$ for all sufficiently long words $h\in 
A^\ast_{-1}$. 
\begin{figure}[h]
\begin{center}
\includegraphics{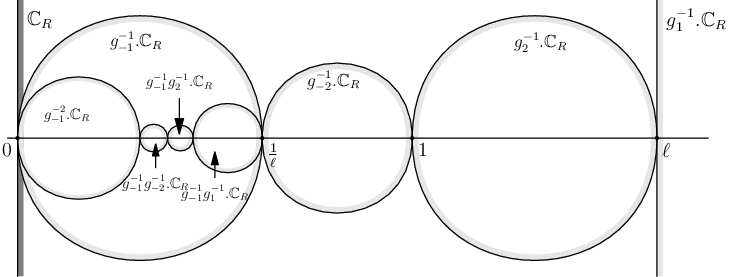}
\end{center}
\caption{Images of $\C_R$ under $A^\ast$ for $q=5$.}
\label{fig:forward}
\end{figure}
\begin{figure}[h]
\begin{center}
\includegraphics{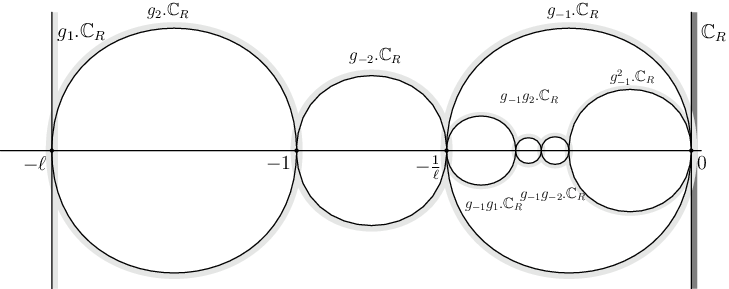}
\end{center}
\caption{Images of $\C_R$ under $A^{-\ast}$ for $q=5$.}
\label{fig:backward}
\end{figure}

Figure~\ref{fig:backward} indicates the location of $g^{-1}.\C_R$ for $g\in 
A^\ast$. Since $\mc U$ and $Q.\mc U$ are bounded away from $(-\infty,0]$ there exists $n\in\N$ such that for all words $g$ in $A^*$ of length at least $n$ we have 
\[
 \mc U, Q.\mc U \subseteq g^{-1}.\C_R.
\]
Thus, $g.\mc U, gQ.\mc U \subseteq \C_R$. Using $g.\mc U$ and $gQ.\mc U$ in place of $\mc W$ in the consideration above shows that for all sufficiently long words $h$ in $A^*_{-1}$ we have $h.\mc U, hQ.\mc U\subseteq \mc U_{-1}$.

\begin{proof}[Proof of Lemma~\ref{lem:allcontracts}]
We only provide a proof for \eqref{lem:alli} as the other statements are seen analogously. We start by showing \eqref{lem:alli} for $\mc U\subseteq\C_R$. Indeed we establish it for $\C_R$ instead of $\mc U$, which is a slightly stronger statement. We remark that $Q.\C_R = \C_R$.

The set
\[
 \mc F^* \sceq \{ z\in\h\mid \Rea z \in (0,\ell),\ |z|>1,\ |z-\ell|>1\}
\]
is a fundamental domain for the action of $\Gamma$ on $\h$. Its vertical sides 
\[
\{z\in\overline{\mc F^*}\mid \Rea z = 0\}\quad\text{and}\quad\{z\in\overline{\mc F^*}\mid \Rea z = \ell\}
\]
are identifies via $T$, and the two bottom sides 
\[
\{ z\in\overline{\mc F^*}\mid |z|=1,\ \Rea z\leq \ell/2\}\quad\text{and}\quad\{z\in\overline{\mc F^*}\mid |z-\ell|=1,\ \Rea z\geq \ell/2\}
\]
are identified via $S$. The set $\mc F^*$ relates to the fundamental domain $\mc F$ in Figure~\ref{funddoms} by shifting its part in $\{\Rea z> \ell/2\}$ by $-\ell$. Let 
\[
 B \sceq \bigcup_{k=1}^q (TS)^q.\overline{\mc F^*}.
\]
We state several properties of the set $B$ and refer for proofs to \cite[Section~4]{Pohl_Spratte} and \cite{Pohl_Symdyn2d}. From the side-pairing properties of $\mc F^*$ it follows that $B$ is the hyperbolic polyhedron (see Figure~\ref{fig:set_B}) with vertices
\begin{align*}
 & \infty,\ g_1^{-1}.0 = g_2^{-1}.\infty,\ g_2^{-1}.0=g_3^{-1}.\infty,\ \ldots,\  g_{m}^{-1}.0=g_{-m}^{-1}.\infty,\\
 & \ldots, g_{-2}^{-1}.0=g_{-1}^{-1}.\infty,\ g_{-1}^{-1}.0=0.
\end{align*}
\begin{figure}[h]
\begin{center}
\includegraphics{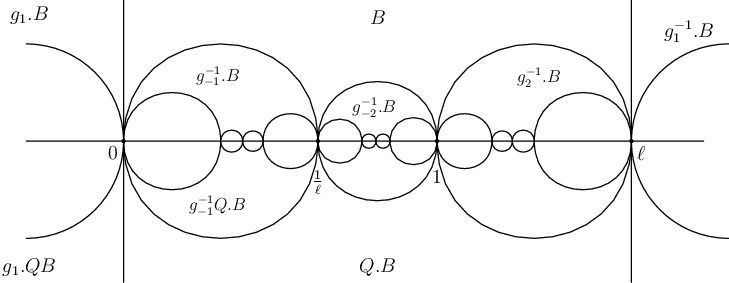}
\end{center}
\caption{The sets $B$ and $Q.B$ and some neighboring translates.}
\label{fig:set_B}
\end{figure}

Further, 
\[
 \Stab_\Gamma(B) = \{ (TS)^k \mid k=1,\ldots, q\}
\]
stabilizes $B$ (as a set), and 
\[
 \{ g.B \mid g\in\Gamma\} = \{ g.B \mid g\in\Gamma/\Stab_\Gamma(B)\}
\]
is a tesselation of $\h$. The neighboring translates of $B$ in $\h\cap \C_R$ are given by $g_j^{-1}.B$ with $j\in\{\pm1,\ldots,\pm m\}$, the overlapping side of $g_j^{-1}.B$ with $B$ is
\[
 B\cap g_j^{-1}.B = g_j^{-1}.(i\R_{>0}).
\]
Note that $Q.B$ is the reflection of $B$ at the real axis. Let
\[
\h^-\sceq \{z\in\C\mid \Ima z < 0\} 
\]
denote the lower half plane, and recall the action of $\Gamma$ on $\C$ as defined in \eqref{action_ext}. Then 
\[
 \{ gQ.B\mid g\in\Gamma\}
\]
is a tesselation of $\h^-$, and the neighboring translates of $Q.B$ in $\h^-\cap \C_R$ are given by $g_j^{-1}Q.B$ with $j\in\{\pm1,\ldots,\pm m\}$. 

Let $E\sceq B\cup Q.B$. The tesselation properties of $B$ and $Q.B$ (and the continuity of the $\Gamma$-action on $\C$) show that 
\[
 \{g.E\mid g\in\Gamma\}
\]
tesselates $\C$, and
\[
 \{ g.E \mid g\in A^*\}
\]
tesselates $\C_R$. Further, the geometric forms of $B$ and $Q.B$, and hence of $E$, yield the following properties:
\begin{enumerate}[(a)]
\item If $h,k\in A^*$ then $hk.\C_R \subsetneqq h.\C_R$. Since the $\Gamma$-action on $\C$ is continuous, this statement holds indeed for $\C_R$, not only for $\C_R\smallsetminus\R$.
\item For $M\subseteq\C$ let $\diam(M)$ denote the diameter of $M$ in the Euclidean metric of $\C$. For any sequence $(h_n)_{n\in\N}$ in $A$ we have
\[
 \diam\big(h_n\cdots h_1.\C_R \big) \stackrel{n\to\infty}{\longrightarrow} 0
\]
unless $(h_n)$ is eventually constant $g_1^{-1}$. Further, unless $(h_n)$ is constant $g_1^{-1}$, for all $n\in\N$, $h_n\cdots h_1.\C_R$ is a Euclidean ball centered at the real axis.
\item Let $\len(h)$ denote the length of $h\in A^*$. Then uniformly for $k\in\{-m,\ldots,-1\}\cup\{2,\ldots, m\}$ we have 
\[
 \diam\big( g_k^{-1}h.\C_R\big) \longrightarrow 0 \qquad \text{as $h\in A^*$, $\len(h)\to\infty$.}
\]
\end{enumerate}
In particular, for $h\in A^*_{-1}$, the set $h.\C_R$ is contained in an $\eps$-rounded neighborhood of $D_{-1}$ with $\eps$ only depending on the length of $h$, and shrinking to $0$ as the length of $h$ goes to $\infty$. Since $\mc U_{-1}$ is $\eps$-rounded for some small $\eps>0$, for all but finitely many $g\in A^*_{-1}$ we have $g.\C_R\subseteq \mc U_{-1}$. This shows the statement for $\mc U\subseteq\C_R$, and it shows that $g.(\mc U_{-1}\cap\C_R)\subseteq \mc U_{-1}\cap \C_R$ for all but finitely many $g\in A^*_{-1}$.

We now show \eqref{lem:alli} for the case that $\mc U$ is not necessarily contained in $\C_R$. To that end we set 
\[
 \C_L \sceq \{z\in\C\mid \Rea z\leq 0\}.
\]
The neighboring $\Gamma$-translates of $E=B\cup Q.B$ in $\C_L$ are given by 
\[
 E^*\sceq g_{-m}.E=g_{-m-1}.E=\ldots=g_{-1}.E=g_1.E=\ldots=g_m.E. 
\]
The sides of $E^*$ are given by 
\[
 g_{-m}.(i\R),\ \ldots,\ g_{-1}.(i\R),\ g_1.(i\R),\ \ldots,\ g_m.(i\R),\ i\R.
\]
For $n\in\N_0$ let 
\[
 A^{-n} \sceq \{ h \mid h^{-1}\in A^n\},
\]
and set
\[
 A^{-*} \sceq \bigcup_{n\in\N_0} A^{-n}.
\]
Arguing analogously to above, we find that uniformly for $k\in\{-m,\ldots,-1\}\cup\{2,\ldots, m\}$,
\[
 \diam\big( g_kh.\C_L\big) \longrightarrow 0 \qquad \text{as $h\in A^{-*}$, $\len(h)\to\infty$.}
\]
For each fixed $m_0\in\N$, uniformly for $0\leq m\leq m_0$ and $k\in\{-m,\ldots,-1\}\cup\{2,\ldots, m\}$ we have
\[
 \diam\big( g_1^mg_kh.\C_L\big) \longrightarrow 0 \qquad \text{as $h\in A^{-*}$, $\len(h)\to\infty$.}
\]
Further, for $m\in\N$ and all $h\in A^{-*}$ we have
\[
 g_1^mh.\C_L \subseteq \{ z\in\C\mid \Rea z\leq m\ell\}.
\]
Thus, since $\mc U$ is bounded away from $(-\infty,0]$, there exists $n_0\in\N$ such that for all $h\in A^{-*}$, $\len(h)>n_0$, 
\[
 \mc U \subseteq \C\smallsetminus h.\C_L.
\]
In turn, for $g\in A^*$, $\len(g)>n_0$, 
\[
 g.\mc U\subseteq \C_R.
\]
This completes the proof.
\end{proof}

For $n\in \N_0$ let
\[
 A^n_{L} \sceq A^n_{-1} \cup A^n_{0}.
\]
Then $A^n_L\cup A^n_LQ$ are the elements that act in $\left(\TO^{\slow,\pm}_s\right)^n$. Set
\[
\C^\ast_R\sceq \{ z\in \C \mid \Rea z > 0\}\quad\text{and}\quad 
\C'\sceq\C\setminus (-\infty,0].
\]
Lemma~\ref{lem:allcontracts} allows us to deduce the maximal domain of holomorphy for the functions in $\SFE_s^{\omega,\hol,\pm}$.

Recall the definitions of $\alpha_s^{(1)}$ and $\alpha_s^{(2)}$ from \eqref{defalpha12}, and recall from Section~\ref{sec:actions} that the maximal domain of holomorphy for $\alpha_s^{(k)}(g^{-1})f(z)$ depends not only on the considered function $f$ and the group element $g\in\Gamma$ but also on the choice of $k\in\{1,2\}$. In Proposition~\ref{prop:slow_extension} below, the restrictions on the domain of holomorphy are indeed forced by the maximal domains of holomorphy for $\alpha_s^{(k)}(g)$, $g\in A_L^n$, $n\in\N_0$.

\begin{prop}\label{prop:slow_extension}
Let $s\in\C$ and $\varphi \in \SFE_{s}^{\omega,\hol,\pm}$. If we use $\alpha_s^{(1)}$ 
for $\alpha_s$ then $\varphi$ extends holomorphically to $\C^\ast_R$ and 
satisfies \eqref{slow_fe} on all of $\C^\ast_R$. If we use $\alpha_s^{(2)}$ for $\alpha_s$ then $\varphi$ 
extends holomorphically to $\C'$ and its extension satisfies \eqref{slow_fe} on $\C'$.
\end{prop}

\begin{proof}
By Lemma~\ref{lem:nhbs} we find a rounded neighborhood $\mc U$ of $(0,1)$ to which $\varphi$ has a holomorphic extension. Without loss of generality, we may assume that for 
$k=1,\ldots, m$,  
\begin{equation}\label{allgood}
g_{-k}^{-1}.\mc U \subseteq \mc U\quad\text{and}\quad g_{-k}^{-1}Q.\mc 
U\subseteq \mc U.
\end{equation}
Thus, the identity theorem of complex analysis implies that 
the functional equation
\[
 \varphi = \TO_s^{\slow,\pm}\varphi = \sum_{k=1}^m \big( \alpha_s(g_{-k}) \pm 
\alpha_s(Qg_{-k}) \big) \varphi
\]
remains valid on all of $\mc U$. Even more, for any $n\in\N$ we have
\begin{equation}\label{slow_iteration}
 \varphi = \left(\TO_{s}^{\slow,\pm}\right)^n\varphi = \Bigg( \sum_{a\in 
A^n_{L}} \alpha_s(a^{-1}) \pm \alpha_s(Qa^{-1}) \Bigg)\varphi
\end{equation}
on $(0,1)$, and hence on $\mc U$. 

Note that for $\alpha_s^{(1)}$ the set $\C^\ast_R$ is the largest domain that contains $(0,1)$ and on which all the cocycles in \eqref{slow_iteration} (for all $n\in\N$) are 
well-defined and holomorphic. For $\alpha_s^{(2)}$, the slit plane $\C'$ is the largest domain with these properties. In case we use $\alpha_s^{(1)}$ let $\mc D\sceq \C^\ast_R$, otherwise let $\mc D\sceq \C'$.

For $z_0\in\mc D$ fix an open bounded neighborhood $\mc W$ of $z_0$ in $\mc D$ that is bounded away from $(-\infty,0]$. By Lemma~\ref{lem:allcontracts} there exists $n_0\in\N$ such that for $n\geq n_0$ and $g\in A^n_L$ we have $g.\mc W\subseteq \mc U$ and $gQ.\mc W \subseteq \mc U$. We fix $n\geq n_0$ and define
\begin{equation}\label{def_ext1}
 \varphi_{\mc W} \sceq \Bigg( \sum_{a\in A^n_{L}} \alpha_s(a^{-1}) \pm \alpha_s(Qa^{-1}) 
\Bigg)\varphi \quad \text{on $\mc W\cup \mc U$.} 
\end{equation}
Note that the right hand side of \eqref{def_ext1} is indeed defined on $\mc W\cup\mc U$ since $\varphi$ is defined on $\mc U$, and $\mc U$ satisfies \eqref{allgood}. 

In order to see that the definition of $\varphi_{\mc W}$ is independent of the choice of $n$ let $m\geq n_0$. Without loss of generality, we may suppose that $m>n$. Using \eqref{slow_iteration} and \eqref{def_ext1} we find on all of $\mc W\cup\mc U$ the identity
\begin{align*}
&\Big( \sum_{a\in A^n_{L}} \alpha_s(a^{-1})  \pm \alpha_s(Qa^{-1}) \Big)\varphi 
\\
& \quad = \Big( \sum_{a\in A^n_{L}} \alpha_s(a^{-1})  \pm \alpha_s(Qa^{-1}) 
\Big) \Big( \sum_{b\in A^{m-n}_{L}} \alpha_s(b^{-1})  \pm \alpha_s(Qb^{-1}) 
\Big) \varphi
\\
& \quad = \Big( \sum_{a\in A^n_L, b\in A^{m-n}_L} \alpha_s(a^{-1}b^{-1}) 
\pm \alpha_s(Qa^{-1}b^{-1}) \pm \alpha_s(a^{-1}Qb^{-1}) + 
\alpha_s(Qa^{-1}Qb^{-1}) \Big)\varphi
\\
& \quad = \Big( \sum_{c\in A^m_L} \alpha_s(c^{-1}) \pm 
\alpha_s(Qc^{-1})\Big)\varphi.
\end{align*}
Thus, $\varphi_{\mc W}$ does not depend on the choice of $n\geq n_0$.

The identity \eqref{slow_iteration} implies immediately that $\varphi_{\mc W}=\varphi$ on $\mc U$. Moreover, if for $j\in\{1,2\}$, $z_j\in\mc D$, $\mc W_{z_j}$ is an open bounded neighborhood of $z_j$ in $\mc D$ bounded away from $(-\infty,0]$, and $\varphi_{\mc W_j}$ is the function defined by \eqref{def_ext1} then the combination of \eqref{def_ext1} with \eqref{slow_iteration} yields that 
\[
 \varphi_{\mc W_1} = \varphi_{\mc W_2}.
\]
From these observations it follows that if we fix for any $z\in\mc D$ an open bounded neighborhood $\mc W_z$ in $\mc D$ bounded away from $(-\infty,0]$, and let $\varphi_z$ denote the function defined by \eqref{def_ext1} then $\psi\colon\mc D\to\C$, 
\[
 \psi(z)\sceq \varphi_z(z)
\]
is a holomorphic extension of $\varphi$ to $\mc D$ which coincides with $\varphi$ on $\mc U$. The identity theorem yields that $\psi$ satisfies \eqref{slow_fe} on all of $\mc D$. 
\end{proof}

Let 
\[
 B \sceq \left\{g_{\pm1}^{-p}, g_{\pm2}^{-1},\ldots, g_{\pm m}^{-1}\ \left\vert\ 
p\in\N \vphantom{g_{\pm1}^{-p}}\right.\right\}.
\]
We call a word over the alphabet $B$ \textit{reduced} if it does not contain a 
subword of the form $g_1^{-p_1}g_1^{-p_2}$ or $g_{-1}^{-p_1}g_{-1}^{-p_2}$ with 
$p_1,p_2\in\N$. For each $n\in\N_0$, let
\[
 B^n \sceq \left\{ h_{k_1}\cdots h_{k_n} \ \left\vert\ \text{$h_{k_j}\in B$ for 
$j=1,\ldots,n$}\right.\right\}
\]
denote the set of reduced words of length $n$ over $B$. Further let
\begin{align*}
B^n_0 &\sceq \left\{ h_{k_1}\cdots h_{k_n}\in B^n \ \left\vert\ 
k_1\in\{-2,\ldots,-m\} \vphantom{h_{k_1}}\right.\right\},
\\
B^n_{(0,1)} &\sceq \left\{ h_{k_1}\cdots h_{k_n} \in B^n_{0} \ \left\vert\ k_n = 
1 \vphantom{h_{k_1}}\right.\right\},
\\
B^n_{-1} &\sceq \left\{ h_{k_1}\cdots h_{k_n}\in B^n \ \left\vert\ k_1=-1 
\vphantom{h_{k_1}}\right.\right\},
\\
B^n_{(-1,-1)} &\sceq \left\{ h_{k_1}\cdots h_{k_n}\in B^n_{-1} \ \left\vert\ 
k_n=-1 \vphantom{h_{k_1}}\right.\right\}
\intertext{and}
B^n_{(-1,1)} &\sceq \left\{ h_{k_1}\cdots h_{k_n}\in B^n_{-1} \ \left\vert\ 
k_n=1 \vphantom{h_{k_1}}\right.\right\}.
\end{align*}
Then these sets determine the elements that act in $\left(\TO^{\fast,\pm}_s\right)^n$, for the exact relation we refer to the proof of Proposition~\ref{prop:fast_extension} below. Lemma~\ref{lem:allcontracts} allows us to determine the maximal domain of holomorphy for the function vectors in $\FFE_s^{\pm}$.

\begin{prop}\label{prop:fast_extension}
Let $s\in \C$ and $f=(f_{0}, f_{-1})^\top \in \FFE_s^\pm$. If we use 
$\alpha_s^{(1)}$ for $\alpha_s$ then $f_{0}$ extends holomorphically to 
$\C^\ast_R$ and $f_{-1}$ extends holomorphically to 
\[
\C^\ast_\ell\sceq \{z\in\C\mid \Rea z > -1/\ell\}.
\]
The holomorphically extended function vector $f=(f_{0}, f_{-1})^\top$ satisfies
\begin{equation}\label{fast_fe}
 f = \begin{pmatrix}
\big(1\pm \alpha_s(Q)\big)\TO_{0,s}^\fast &  \big(1\pm 
\alpha_s(Q)\big)\TO_{-1,s}^\fast
\\[1mm]
\big(1\pm \alpha_s(Q)\big)\TO_{0,s}^\fast & \pm\alpha_s(Q)\TO_{-1,s}^\fast
\end{pmatrix}
f.
\end{equation}
If we use $\alpha_s^{(2)}$ for $\alpha_s$ then $f_{0}$ extends holomorphically 
to $\C'$ and $f_{-1}$ extends holomorphically to $\C\setminus (-\infty, 
-1/\ell]$, and the function vector $(f_{0}, f_{-1})^\top$ satisfies 
\eqref{fast_fe}.
\end{prop}

\begin{proof}
It suffices to show the statement for $\Rea s > 1/2$. We only provide the 
proof for $\alpha_s^{(1)}$ as the consideration of $\alpha_s^{(2)}$ is 
analogous. We note that $\C^\ast_R \times \C^\ast_\ell$ is the maximal domain of 
holomorphy that contains $D_{0} \times  D_{-1}$ and on which all arising 
cocycles are simultaneously well-defined and holomorphic. 

For $n\in \N_0$ we have (see \cite[Lemma~5.2]{Pohl_spectral_hecke}; note that the notation here is different and that the $Q$-contributions are handled in a different, though equivalent, way; alternatively it follows from \cite[Proof of Proposition~4.11]{Moeller_Pohl} where one still needs to perform the passage from the transfer operator for $\Gamma$ to the pair of transfer operators for $\wt\Gamma$ as in \cite[Proposition~4.15]{Moeller_Pohl})

{\scalebox{.88}
{
\begin{minipage}{\the\textwidth}
\begin{align*}
\left(\TO_s^{\fast,\pm}\right)^n 
= 
\begin{pmatrix}
(1\pm \alpha_s(Q)) \sum\limits_{\scalebox{\verklein}{$b\in B^n_{0}$}} 
\alpha_s(b^{-1}) & (1\pm\alpha_s(Q)) \sum\limits_{{\scalebox{\verklein}{$b\in 
B^n_{-1}$}}}\alpha_s(b^{-1})
\\[4mm]
\sum\limits_{\scalebox{\verklein}{$b\in B^n_{0}\setminus 
B^n_{(0,-1)}$}}\hspace*{\verschieb} \alpha_s(b^{-1}) \pm \hspace*{\verschiebt} 
\sum\limits_{\scalebox{\verklein}{$b\in B^n_{0}\setminus 
B^n_{(0,1)}$}}\hspace*{\verschieb} \alpha_s(Qb^{-1}) 
& 
\sum\limits_{\scalebox{\verklein}{$b\in B^n_{-1}\setminus 
B^n_{(-1,-1)}$}}\hspace*{\verschieb}\alpha_s(b^{-1}) \pm \hspace*{\verschiebt} 
\sum\limits_{\scalebox{\verklein}{$b\in B^n_{-1}\setminus 
B^n_{(-1,1)}$}}\hspace*{\verschieb}\alpha_s(Qb^{-1})
\end{pmatrix}.
\end{align*}
\end{minipage}
}
}

Let $(z_0,w_0)\in \C^\ast_R\times \C^\ast_\ell$ and pick open bounded 
neighborhoods $\mc U$ of $z_0$ in $\C^\ast_R$ and $\mc V$ of $w_0$ in 
$\C^\ast_\ell$. Further, for $j\in\{-1,0\}$, let $\mc D_j$ be open complex 
neighborhoods of $\overline{D_j}$ such that $f\in B(\mc D_0) \oplus B(\mc 
D_{-1})$. Note that $\mc D_j$ is a rounded neighborhood of $D_j$ for $j\in\{-1,0\}$.

By Lemma~\ref{lem:allcontracts} there exists $n_0\in\N$ such that for $n\geq 
n_0$ we have 
\[
 g.\mc U \subseteq \mc D_{0}\quad\text{and}\quad gQ.\mc U \subseteq \mc D_{0}
\]
for all $g\in B^n_{0}$, and
\[
 g.\mc V \subseteq \mc D_{-1}\quad\text{and}\quad gQ.\mc V \subseteq \mc D_{-1}
\]
for all $g\in B^n_{-1}$. We fix $n\geq n_0$ and define
\begin{equation}\label{def_ext2}
\begin{pmatrix}
  f_{0} \\ f_{-1}
\end{pmatrix}
\sceq\left(\TO_s^{\fast,\pm}\right)^n
\begin{pmatrix}
  f_{0} \\ f_{-1}
\end{pmatrix}
\end{equation}
on $\mc U\times \mc V$. As in the proof of Proposition~\ref{prop:slow_extension} 
we see that the left hand side of \eqref{def_ext1} is well-defined and defines a 
holomorphic function vector that satisfies \eqref{fast_fe} on $\mc U \times \mc 
V$.
\end{proof}

\subsubsection{A crucial identity}\label{sec:crucial}
In this section we show that
\[
 \TO_{-1,s}^\fast f_{-1} = \alpha_s(g_{-1})\varphi \quad\text{on $\R_{>0}$}
\]
whenever $f=(f_0,f_{-1})^\top \in \FFE_s^{\pm}$ is given and $\varphi$ is 
defined by \eqref{ftophi}, or $\varphi\in \SFE^{\omega,\as,\pm}_s$ is given and 
$f$ is defined by \eqref{phitof}. More precisely, we show that 
\begin{equation}\label{identity_1}
 \alpha_s(g_{-1})\circ\left(1+\TO_{-1,s}^\fast\right)f_{-1} = 
\TO_{-1,s}^\fast f_{-1}
\end{equation}
and 
\begin{equation}\label{identity_2}
 \TO_{-1,s}^\fast\circ\big( 1- \alpha_s(g_{-1})\big)\varphi = \alpha_s(g_{-1})\varphi
\end{equation}
on $\R_{>0}$. Furthermore we provide regularity properties which allow us to determine the spaces between which \eqref{ftophi} and \eqref{phitof} establish isomorphisms.

A crucial tool for these investigations are asymptotics of the Lerch zeta function $\zeta(s,a,x)$ (see Section~\ref{sec:meromorphic}) for large values of $x$. Since we consider it here for $x>0$ only, we have $\alpha_s=\alpha_s^{(1)}=\alpha_s^{(2)}$ and thus do not need to distinguish between the two variants of the (meromorphically continued) Lerch zeta function. Its asymptotic expansion for $x\to\infty$ is 
\begin{equation}\label{lerch_exp}
 \zeta(s,a,x) \sim \sum_{n=-1}^\infty D_n x^{-(s+n)}
\end{equation}
for certain coefficients $D_n\in\C$, $n\in\Z_{\geq-1}$, depending on $s$ and $a$  
with $D_{-1}=0$ if $a\notin\Z$ \cite{Katsurada}. The precise (numerical) expressions for all 
$D_n$ are known \cite{Katsurada} but they are not of importance to us.

\begin{prop}\label{prop:seq}
Let $s\in\C$ and $f=(f_0,f_{-1})^\top\in \FFE_s^\pm$. Then
\begin{enumerate}[{\rm (i)}]
\item\label{seq1} $\alpha_s(g_{-1})\circ\left(1+\TO_{-1,s}^\fast\right)f_{-1} = 
\TO_{-1,s}^\fast f_{-1}$ on $\R_{>0}$.
\item\label{seq2} $\left(1+\TO_{-1,s}^\fast\right)f_{-1}(x) = \frac{c}{x} + 
O_{x\to0^+}(1)$ for some $c=c(s,f)\in V$. Moreover, $\pr_r(c) = 0$.
\end{enumerate}
\end{prop}

\begin{proof}
To simplify notation, we set  $\TO_s \sceq \TO_{-1,s}^\fast$. We start with a 
diagonalization. Since $\chi(g_{-1})$ is a unitary operator on $V$, there 
exists an orthonormal basis of $V$ with respect to which $\chi(g_{-1})$ is 
represented by a unitary diagonal matrix, say
\[
 \diag\left(e^{2\pi i a_1},\ldots, e^{2\pi i a_d} \right)
\]
with $a_1,\ldots, a_d\in \R$ and $d=\dim V$. We use the same basis of $V$ to 
represent any function $\psi\colon D \to V$ (here, $D$ is any domain that arises 
in our considerations) as a vector of component functions
\[
 \begin{pmatrix} \psi_1\\ \vdots \\ \psi_d\end{pmatrix} \colon D\to \C^d.
\]
For $s\in\C$, $g\in G$, any subset $I$ of $\R$ and any function $f\colon I\to 
\C$ we set
\begin{equation}\label{def_taus}
 \tau_s(g^{-1})f(x) \sceq |g'(x)|^s f(g.x),
\end{equation}
whenever it makes sense. Then, in these coordinates for $V$ and for $\Rea s > 
\tfrac12$, the operator $\TO_s$ acts as
\[
 \diag\left( \sum_{n\in\N} e^{2\pi i n a_1} \tau_s(g_{-1}^n),\ldots, 
\sum_{n\in\N} e^{2\pi i n a_d}\tau_s(g_{-1}^n) \right).
\]
We now consider a single component. Let $a\in\R$ and, by a slight abuse of 
notation, set
\[
 \alpha_s(g_{-1}) \sceq \alpha_s^\C(g_{-1}) \sceq e^{2\pi i a}\tau_s(g_{-1}).
\]
For $\Rea s > \tfrac12$ let
\begin{equation}\label{Ls}
 L_s \sceq \sum_{n\in\N} \alpha_s(g_{-1}^n)= \sum_{n\in\N} e^{2\pi i 
na}\tau_s(g_{-1}^n),
\end{equation}
and let $h$ be a real-analytic complex-valued function that is defined in some 
neighborhood of $0$. For $k\in\N_0$ let
\begin{equation}\label{taylor}
 c_k \sceq \frac{h^{(k)}(0)}{k!\ell^k} \quad\text{and}\quad h_k(x) \sceq c_k \ell^k x^k = \frac{h^{(k)}(0)}{k!}x^k.
\end{equation}
Let $M\in\N_0$. In order to state $L_s$'s meromorphic continuation to $\Rea s > 
(1-M)/2$ we define
\[
 P_M(h)(x) \sceq h(x) - \sum_{k=0}^{M-1} h_k(x)
\]
and $Q_M\sceq 1 - P_M$. Then
\[
 L_s = L_s \circ Q_M + L_s \circ P_M,
\]
where $L_s\circ P_M$ converges for $\Rea s >(1-M)/2$ and the meromorphic 
continuation of $L_s\circ Q_M$ is given by
\[
 \left(L_s\circ Q_M\right)h\colon x \mapsto \frac{e^{2\pi i a}}{(\ell 
x)^{2s}}\sum_{k=0}^{M-1} c_k \zeta\left(2s+k, a, 1+\frac{1}{\ell x}\right).
\]
For the proof of \eqref{seq1} note that 
\[
 \left(\alpha_s(g_{-1})\circ L_s \circ Q_M\right)h(x) = \frac{e^{2\pi i 
2a}}{(\ell x)^{2s}}\sum_{k=0}^{M-1} c_k \zeta\left(2s+k, a, 2+\frac{1}{\ell 
x}\right)
\]
and
\begin{align*}
\left( \alpha_s(g_{-1})\circ L_s \circ P_M\right)h &= L_s \circ P_M h+ L_s \circ 
Q_M h - \alpha_s(g_{-1})P_Mh - L_s\circ Q_M h.
\end{align*}
Thus,
\begin{align*}
\alpha_s(g_{-1}) L_s h (x) & = \alpha_s(g_{-1})L_s P_M h(x) + 
\alpha_s(g_{-1})L_s Q_M h(x)
\\
& = L_s h(x) - \alpha_s(g_{-1})h(x) + \sum_{k=0}^{M-1} \frac{c_k e^{2\pi i 
a}}{(\ell x)^{2s}} \Bigg[ \left( 1 + \frac{1}{\ell x}\right)^{-(2s+k)}
\\
& \qquad- \zeta\left(2s+k, a, 1+\frac{1}{\ell x}\right) + e^{2\pi i 
a}\zeta\left(2s+k,a,2+\frac{1}{\ell x}\right) \Bigg]
\\
& =  L_s h(x) - \alpha_s(g_{-1})h(x).
\end{align*}
This proves \eqref{seq1}.

For \eqref{seq2} we claim that there exists an asymptotic expansion of the form
\begin{align}\label{asympexp2}
 (1+L_s) h(x) &\sim  \sum_{p=-1}^\infty c_p^* x^p \qquad\text{as $x\to 0^+$}
\end{align}
with complex coefficients $c_p^*$ (depending on $s,a,h$) for $p\in\Z_{\geq -1}$ such that $c_{-1}^*=0$ if 
$a\notin\Z$. Then \eqref{seq2} immediately follows from \eqref{asympexp2}.

Let 
\[
 K_s\sceq 1+L_s
\]
and recall that \eqref{asympexp2} means by definition that for each $P\in\Z_{\geq -1}$ we have
\[
 K_sh(x) - \sum_{p=-1}^P c_p^*x^p = o\big(x^P\big) \qquad\text{as $x\to 0^+$}.
\]
In order to establish \eqref{asympexp2} let $P\in\Z_{\geq -1}$, pick $M\in\N_0, M\geq P+2$ such that $\Rea s > (1-M)/2$, and consider the splitting 
\begin{align*}
 K_sh(x) & = (K_s\circ Q_M)h(x) + (K_s\circ P_M)h(x).
\end{align*} 
In the following we first prove that 
\begin{equation}\label{restsmall}
 (K_s\circ P_M)h(x) = o\big(x^{M-2}\big) \qquad\text{as $x\to0^+$.}
\end{equation}
Then we show that $(K_s\circ Q_M)h(x)$ has an asymptotic expansion of the form \eqref{asympexp2}, and that its first $P+1$ coefficients (that is, those for the terms $x^{-1}, \ldots, x^{P}$) do not depend on the choice of $M$. These two results immediately imply \eqref{asympexp2}. Their proofs even provide an exact formula for the coefficients in the asymptotic expansion, see \eqref{coeff} below. 

We first note that \eqref{lerch_exp} implies for each $k\in\N_0$ the asymptotic expansion (recall $c_k$ from \eqref{taylor})
\begin{equation}\label{lerch_asymp2}
 (\ell x)^{-2s}c_k\zeta\left(2s+k,a,\frac{1}{\ell x}\right) \sim \sum_{n=-1}^\infty D_n(k)x^{k+n} \qquad\text{as $x\to 0^+$} 
\end{equation}
for certain coefficients $D_n(k)\in\C$, $n\in\Z_{\geq -1}$, depending on $s$ and $a$, and with $D_{-1}(k)=0$ if $a\notin\Z$. In particular, 
\begin{equation}\label{growth2}
 (\ell x)^{-2s}c_k\zeta\left(2s+k,a,\frac{1}{\ell x}\right) = o\big( x^{k-2} \big) \qquad\text{as $x\to0^+$.}
\end{equation}

In order to show \eqref{restsmall}, we recall that the Taylor formula with Lagrange remainder term yields that for each $n\in\N_0$ and $x>0$ there exist vectors
\[
\xi_R(x,n)=\xi_R(x,n,M),\ \xi_I(x,n)=\xi_I(x,n,M)\in \left( 0, \frac{x}{n\ell x+1}\right)^{\dim V}
\]
such that 
\[
 (P_Mh)\left( \frac{x}{n\ell x +1}\right) = \frac{\Rea h^{(M)}\big( \xi_R(x,n) \big) + i\Ima h^{(M)}\big(\xi_I(x,n)\big)}{M!} \cdot \left( \frac{x}{n\ell x +1}\right)^M.
\]
Thus, 
\begin{align*}
 ( K_s\circ P_M)h(x) &= \sum_{n=0}^\infty \frac{e^{2\pi i n a}}{(n\ell x +1)^{2s}} (P_Mh)\left(\frac{x}{n\ell x +1}\right)
 \\
 & =  (\ell x)^{-2s}  \sum_{n=0}^\infty e^{2\pi i na} \left( n+ \frac{1}{\ell x}\right)^{-(2s+M)} \cdot c(x,n),
\end{align*}
where 
\[
 c(x,n)\sceq \frac{\Rea h^{(M)}\big( \xi_R(x,n)\big) + i\Ima h^{(M)}\big(\xi_I(x,n)\big)}{M!}.
\]
Since $\xi_R(x,n)$ and $\xi_I(x,n)$ are bounded uniformly in $x$ and $n$, so is $c(x,n)$. It follows that 
\begin{align*}
  \left|( K_s\circ P_M)h(x)\right| & \ll_M (\ell x)^{-2\Rea s} \sum_{n=0}^\infty \left( n+ \frac{1}{\ell x}\right)^{-(2\Rea s+M)} 
  \\
  & \quad = (\ell x)^{-2\Rea s} \zeta\left( 2\Rea s + M, 0, \frac{1}{\ell x}\right) 
\end{align*}
for all $x>0$, with implied constant independent of $x$. Thus, \eqref{growth2} implies \eqref{restsmall}.

We now investigate $(K_s\circ Q_M)h(x)$. For all $x>0$ we have
\begin{align*}
(K_s\circ Q_M)h(x)  & = (\ell x)^{-2s} \sum_{k=0}^{M-1} c_k \zeta\left(2s+k,a,\frac{1}{\ell x}\right).
\end{align*}
Thus, it follows from \eqref{lerch_asymp2} that $(K_s\circ Q_M)h(x)$ has the asymptotic expansion 
\begin{align}\label{part_exp}
(K_s\circ Q_M)h(x)  &  \sim \sum_{k=0}^{M-1} \sum_{n=-1}^\infty D_n(k) (\ell x)^{k+n} 
 = \sum_{p=-1}^\infty \wt c_p(M) x^p 
\end{align}
as $x\to0^+$, where 
\begin{equation}\label{specialcoeff}
 \wt c_p(M) \sceq  \ell^p \sum_{k=0}^{M-1}\sum_{n=-1}^\infty \delta_{k+n,p} D_n(k) 
\end{equation}
for all $p\in\Z_{\geq -1}$. Here, 
\[
 \delta_{q,p} \sceq 
 \begin{cases}
  1 & \text{if $p=q$}
  \\
  0 & \text{if $p\not=q$}
 \end{cases}
\]
denotes the Dirac $\delta$-function. Note that for each $p\in\Z_{\geq -1}$, the series in \eqref{coeff} has only finitely many non-vanishing summands, and hence it is indeed a finite sum.

If $a\notin\Z$ then 
\[
 \wt c_{-1}(M) = D_{-1}(0) = 0.
\]
Thus, the asymptotic expansion \eqref{part_exp} is indeed of the form \eqref{asympexp2}. Further, \eqref{specialcoeff} shows that for $p\leq P$ we have
\[
\wt c_p(M) = \ell^p \sum_{q=-1}^{p} D_q(p-q),
\]
which is indeed independent of the choice of $M\geq P+2$.

This completes the proof of the existence of the asymptotic expansion \eqref{asympexp2}, and it furthermore shows that for $p\in\Z_{\geq -1}$ the coefficient $c_p^*$ is given by
\begin{equation}\label{coeff}
 c_p^* \sceq \sum_{k=0}^\infty \sum_{n=-1}^\infty \delta_{k+n,p} D_n(k).\qedhere
\end{equation}
\end{proof}

\begin{prop}\label{prop:ae_generic}
Let $s\in\C$ and $\varphi\in \SFE_s^{\omega,\pm}$. Set
\begin{align}\label{defpsi}
 \psi&\sceq \big(1-\alpha_s(g_{-1})\big)\varphi = \TO_s^{\slow,\pm}\varphi - 
\alpha_s(g_{-1})\varphi 
\\
& \ = \left( \big(1\pm\alpha_s(Q)\big)\sum_{k=2}^m \alpha_s(g_{-k}) \pm \alpha_s(Qg_{-1}) \right)\varphi \nonumber
\end{align}
Then 
\[
 \Phi\sceq \Phi_{s,\varphi} \sceq \alpha_s(g_{-1})\varphi - \TO_{-1,s}^{\fast}\psi \colon \R_{>0}\to V
\]
is a real-analytic $\alpha_s(g_{-1})$-invariant function. Further, $\varphi$ has 
an asymptotic expansion of the form
\begin{equation}\label{asymexp}
 \varphi(x) \sim \Phi(x) + \sum_{n=-1}^\infty C^*_nx^n \qquad\text{as $x\to 0^+$}
\end{equation}
for certain (unique) coefficients $C^*_n\in V$, $n\in\Z_{\geq -1}$. Moreover, 
$\pr_r(C^*_{-1})=0$.
\end{prop}

\begin{proof}
Obviously, $\psi$ extends real-analytically to some neighborhood of $0$, and hence $\Phi$ is real-analytic. We start by showing that $\Phi$ is 
$\alpha_s(g_{-1})$-invariant. To that end let $f$ be an arbitrary function which 
is smooth in a neighborhood of $0$. To simplify notation, we set 
\[
 \TO_s \sceq \TO_{-1,s}^\fast.
\]
For $\Rea s >\tfrac12$ we have
\begin{equation}\label{shiftLs}
\alpha_s(g_{-1})\TO_s f  = \TO_s f - \alpha_s(g_{-1})f.
\end{equation}
Since $f$ is arbitrary (hence, in particular, independent of $s$), meromorphic 
continuation in $s$ shows that \eqref{shiftLs} holds for all 
$s\in\C\setminus\{\text{poles}\}$. Thus, applying \eqref{shiftLs} with $f=\psi$ 
and recalling \eqref{defpsi} yields 
\begin{align*}
 \alpha_s(g_{-1})\Phi & = \alpha_s(g_{-1}^2)\varphi - \alpha_s(g_{-1})\TO_s\psi 
 \\
 & = \alpha_s(g_{-1}^2)\varphi - \TO_s\psi + \alpha_s(g_{-1})\psi
 \\
 & = \alpha_s(g_{-1}^2)\varphi - \TO_s\psi + \alpha_s(g_{-1})\varphi - 
\alpha_s(g_{-1}^2)\varphi
 \\
 & = - \TO_s\psi + \alpha_s(g_{-1})\varphi 
 \\
 & = \Phi.
\end{align*}
Hence, $\Phi$ is $\alpha_s(g_{-1})$-invariant.

For the asymptotic expansion we note that 
\begin{equation}\label{phi_eq}
 \varphi = \Phi + \psi + \TO_s\psi.
\end{equation}
From
\[
 \psi = \big(1\pm \alpha_s(Q)\big)\sum_{k=2}^m\alpha_s(g_{-k})\varphi \pm \alpha_s(Qg_{-1})\varphi
\]
and the fact that for $k\in\{2,\ldots,m\}$ the elements $g_{-1}^{-1}Q, g_{-k}^{-1}, g_{-k}^{-1}Q$ map (small) neighborhoods of $0$ away from $0$ it follows that $\psi$ extends to a real-analytic function in a neighborhood of $0$. As in the proof of Proposition~\ref{prop:seq} we find that the asymptotic 
expansion of $\psi+\TO_s\psi$ for $x\to0^+$ is of the claimed form.
\end{proof}

\begin{lemma}\label{lem:Q0}
Let $s\in\C$ and $\varphi\in\SFE_s^{\omega,\pm}$, and let $\Phi=\Phi_{s,\varphi}$ be as in Proposition~\ref{prop:ae_generic}. Then we 
have
\begin{enumerate}[{\rm (i)}]
\item\label{asymp0} If $\Rea s > \tfrac12$ and $\varphi=o_{x\to0^+}(x^{-2s})$ then $\Phi_{s,\varphi}=0$.
\item\label{asymp1} $\Phi_{s,\varphi}(x) = O_{x\to0^+}(x^{-2s})$.
\item\label{asymp2} If $\Phi_{s,\varphi}(x) = o_{x\to0^+}(x^{-2s})$ then $\Phi_{s,\varphi}=0$.
\item\label{asymp3} Let $\tfrac12\geq\Rea s>0$, $s\not=\tfrac12$. If for some $c\in V$,
\begin{equation}\label{Q0_asymp}
 \Phi(x) = \frac{c}{x} + O(1) \qquad\text{as $x\to0^+$} 
\end{equation}
then $c=0$.
\end{enumerate}
\end{lemma}

\begin{proof}
For \eqref{asymp0} recall that, for $\Rea s > \tfrac12$, the operator 
$\TO_{-1,s}^\fast$ is given by \eqref{TOm1_def}. From the decay property of $\varphi$ it follows for all $x\in\R_{>0}$ that
\begin{align*}
\lim_{N\to\infty} \alpha_s\big(g_{-1}^N\big)\varphi(x)
= x^{-2s}\lim_{N\to\infty} \chi\big(g_{-1}^N\big) \left( \frac{x}{x\ell N +1}\right)^{2s} \varphi\left(\frac{x}{x\ell N+1}\right) = 0.
\end{align*}
Thus, $\TO_{-1,s}^\fast\varphi=\alpha_s(g_{-1})\varphi$, and hence $\Phi=0$.

The $\alpha_s(g_{-1})$-invariance of $\Phi$ easily implies \eqref{asymp2}. For 
\eqref{asymp1} and \eqref{asymp3} note that the map
\[
 \wt \Phi_{s,\varphi} \sceq \alpha_s(Q) \Phi_{s,\varphi} \colon (1,\infty) \to \C
\]
is a real-analytic $\alpha_s(g_1)$-invariant function (recall that $Qg_{-1}Q=g_1$). In particular, $\wt \Phi$ 
is bounded. Thus,
\[
 \Phi_{s,\varphi}(x) = \alpha_s(Q)\wt \Phi_{s,\varphi}(x) = x^{-2s} \wt \Phi_{s,\varphi}\left(\frac{1}{x}\right) \ll 
|x^{-2s}|.
\]
This proves \eqref{asymp1}. For \eqref{asymp3} note that \eqref{Q0_asymp} is 
equivalent to 
\begin{equation}\label{growth}
 \wt \Phi_{s,\varphi}(x) = cx^{1-2s} + O(x^{-2s}) \qquad\text{as $x\to\infty$.}
\end{equation}
Thus, for $\tfrac12>\Rea s > 0$ it follows that $\wt \Phi_{s,\varphi}$ is unbounded unless 
$c=0$. Hence the boundedness of $\wt \Phi_{s,\varphi}$ implies $c=0$. It remains to consider 
the case  $\Rea s =\tfrac12$. Let 
\[
 t\sceq -2\Ima s 
\]
and note that $t\not=0$. The $\alpha_s(g_1)$-invariance of $\wt \Phi_{s,\varphi}$ shows that 
for each $x\in (1,\infty)$ and $k\in\N$ we have
\begin{align*}
 |c| \left| x^{it} - (x+k\ell)^{it} \right| & \leq \left| \wt \Phi_{s,\varphi}(x) - 
cx^{it}\right| + \left| \wt \Phi_{s,\varphi}(x+k\ell) - c(x+k\ell)^{it}\right|.
\end{align*}
Thus, the growth condition~\eqref{growth} yields that 
\begin{equation}\label{growthconvergence}
|c| \left| x^{it} - (x+k\ell)^{it} \right| \to 0 \quad \text{as $x\to\infty$, 
$k\to\infty$.}
\end{equation}
We have
\begin{align*}
 \left| x^{it} - (x+k\ell)^{it} \right| = \left| \exp\left(-it\log 
\left(1+\frac{k}{x}\ell\right)\right) - 1  \right|.
\end{align*}
For all $k_0\in\N$, $x_0>1$,
\[
 \left\{ \frac{k}{x} \ \left\vert\ k\geq k_0,\ x\geq x_0 
\vphantom{\frac{k}{x}}\right.\right\} = (0,\infty).
\]
Hence, 
\[
 \limsup_{x\to\infty,k\to\infty} \left| \exp\left(-it\log 
\left(1+\frac{k}{x}\ell\right)\right) - 1  \right| = 2.
\]
In turn, the convergence \eqref{growthconvergence} is only possible for $c=0$. 
This completes the proof.
\end{proof}

\begin{cor}\label{cor:equal}
Let $s\in \C$, $\Rea s > 0$, $s\not=1/2$. Suppose that $\varphi\in 
\SFE_s^{\omega,\as,\pm}$ and define $\psi$ as in \eqref{defpsi}. Then 
\[
 \alpha_s(g_{-1}) \varphi = \TO_{-1,s}^\fast\psi
\]
on $\R_{>0}$.
\end{cor}

\begin{proof}
The combination of Lemma~\ref{lem:Q0} with the asymptotic expansion for 
$\varphi$ from Proposition~\ref{prop:ae_generic} and the growth of 
$\varphi$ towards $0$ immediately yields a proof. 
\end{proof}

The proof of Corollary~\ref{cor:equal} also shows that the elements in $\SFE_s^{\omega,\as,\pm}$ satisfy a stronger condition for the asymptotics as $x\to 0^+$ than requested in their definition, see \eqref{def_sfe} and Remark~\ref{rem:stronger}. 

\begin{cor}\label{cor:vanish}
Let $s\in\C$, $\Rea s > 0$, $s\not=1/2$. Then 
\[
 \SFE_s^{\omega,\as,\pm} = \left\{ \varphi\in \SFE_{s}^{\omega,\hol,\pm} \ 
\left\vert\ \exists\, c\in V, \pr_r(c)=0\colon \varphi(x) = \frac{c}{x} + O_{x\to0^+}(1) 
\right.\right\}.
\]
\end{cor}

\subsubsection{Proof of Theorem~\ref{thm:main_finite}}\label{sec:proofmain}
Suppose first that $\varphi\in \SFE_s^{\omega,\as,\pm}$ and define $f=(f_{0}, 
f_{-1})^\top$ as in \eqref{phitof}. By Proposition~\ref{prop:slow_extension}, 
$\varphi$ extends holomorphically to $\C^\ast_R$ and satisfies \eqref{slow_fe} 
on $\C^\ast_R$. Thus, the definition of $f_{0}$ extends holomorphically to 
$\C^\ast_R$. Further, taking advantage of \eqref{slow_fe}, we find that 
\[
 f_{-1} = \big( 1-\alpha_s(g_{-1})\big)\varphi = \sum_{k=2}^m \big( 
\alpha_s(g_{-k}) \pm \alpha_s(Qg_{-k})\big)\varphi \pm \alpha_s(Qg_{-1})\varphi
\]
is in fact defined and holomorphic on $\C^\ast_\ell$. By the identity theorem of 
complex analysis, it suffices to show that $f$ satisfies $f=\TO_s^{\fast,\pm}f$ 
on $D_{0}\times D_{-1}$. Corollary~\ref{cor:equal} shows $\TO_{-1,s}^\fast 
f_{-1}=\alpha_s(g_{-1})\varphi$ on $\R_{>0}$.

In particular,
\[
 \big(1\pm\alpha_s(Q)\big) \TO_{-1,s}^\fast f_{-1} = \big( \alpha_s(g_{-1}) \pm 
\alpha_s(Qg_{-1})\big)\varphi.
\]
Analogously, on all of $\R_{>0}$ we have
\begin{align*}
\big(1\pm\alpha_s(Q)\big)\TO_{0,s}^\fast f_{0} & = \big(1\pm \alpha_s(Q)\big) 
\TO_{0,s}^\fast\varphi
\\
& = \TO_s^{\slow,\pm}\varphi - \big( \alpha_s(g_{-1}) \pm 
\alpha_s(Qg_{-1})\big)\varphi.
\end{align*}
Then a straightforward calculation shows
\[
 \TO_s^{\fast,\pm} f = f.
\]
If $\varphi$ satisfies \eqref{phi_decay} then $f$ obviously satisfies 
\eqref{f_decay}.

Suppose now that $f=(f_0,f_{-1})^\top\in\FFE_s^\pm$ and define $\varphi$ as in 
\eqref{ftophi}. Since $f_{0}$ and $f_{-1}$ are holomorphic in a complex 
neighborhood of $\overline{D_{0}}$ respectively of $\overline{D_{-1}}$, 
$\varphi$ is real-analytic on $(0,1)$ and even holomorphic in a complex 
neighborhood of $(0,1)$ that is rounded at $0$. Therefore it suffices to show that $\varphi$ satisfies 
$\varphi=\TO_s^{\slow,\pm}\varphi$ on $D_{-1}\cup D_{0}$. By 
Proposition~\ref{prop:seq}\eqref{seq1} we have 
$\alpha_s(g_{-1})\varphi=\TO_{-1,s}^\fast f_{-1}$ on $\R_{>0}$. Then 
$f=\TO_s^{\fast,\pm}f$ yields that on $D_{0}$,
\begin{align*}
\varphi\vert_{D_{0}} & = f_{0} = \big(1\pm\alpha_s(Q)\big) \TO_{0,s}^\fast f_{0} 
+ \big(1\pm\alpha_s(Q)\big)\TO_{-1,s}^\fast f_{-1}
\\
& = \big(1\pm\alpha_s(Q)\big) \sum_{k=2}^m \alpha_s(g_{-k})\varphi + 
\big(1\pm\alpha_s(Q)\big)\alpha_s(g_{-1})\varphi
\\
& = \TO_s^{\slow,\pm}\varphi.
\end{align*}
On $D_{-1}$ we have
\begin{align*}
\varphi\vert_{D_{-1}} & = f_{-1} + \TO_{-1,s}^\fast f_{-1}
\\
& = \big(1\pm\alpha_s(Q)\big) \TO_{0,s}^\fast f_{0} \pm \alpha_s(Q) 
\TO_{-1,s}^\fast f_{-1} + \TO_{-1,s}^\fast f_{-1}
\\
& = \TO_{s}^{\slow,\pm}\varphi.
\end{align*}
This shows $\TO_{s}^{\slow,\pm}\varphi=\varphi$. Then 
Proposition~\ref{prop:seq}\eqref{seq2} yields $\varphi\in 
\SFE_s^{\omega,\as,\pm}$. Finally, if $f$ satisfies \eqref{f_decay} then 
$\varphi$ clearly satisfies \eqref{phi_decay}. This completes the proof of the isomorphism between $\SFE_s^{\omega,\as,\pm}$ and $\FFE_s^\pm$.

In order to prove that this isomorphism descends to an isomorphism between $\SFE_s^{\omega,\dec,\pm}$ and $\FFE_s^{\dec,\pm}$ it suffices to show that $\SFE_s^{\omega,\dec,\pm}\subseteq \SFE_s^{\omega,\hol,\pm}$. To that end let $\varphi\in\SFE_s^{\omega,\dec,\pm}$, and recall the asymptotic expansion \eqref{asymexp} of $\varphi$ as $x\to0^+$. We use the notation from Proposition~\ref{prop:ae_generic}. Remark~\ref{rem:bdd} implies that $C_{-1}^*=0$ and that $\lim_{x\to0^+}\Phi_{s,\varphi}(x)$ exists and equals $C_0^*$. Then the $\alpha_s(g_{-1})$-invariance of $\Phi_{s,\varphi}$ yields that for all $x\in\R$,
\[
 \Phi_{s,\varphi}(x) = \lim_{N\to\infty} \alpha_s\big(g_{-1}^N\big)\Phi_{s,\varphi}(x) = \lim_{N\to\infty} \big( N\ell x + 1\big)^{-2s} \chi\big(g_{-1}^N\big)\Phi_{s,\varphi}\left(\frac{x}{N\ell x +1 }\right) = 0.
\]
Thus, $\varphi = \psi + \TO_s\psi$. Hence $\varphi$ extends holomorphically to a complex neighborhood of $(0,1)$ that is rounded at $0$, and therefore $\varphi\in \SFE_s^{\omega,\hol,\pm}$. This completes the proof of Theorem~\ref{thm:main_finite}. \qed

\subsection{Isomorphism for the Hecke triangle groups $\Gamma_\ell$ with 
$\ell=2\cos(\pi/q)$, $q\geq 4$ even}\label{sec:iso_even}

We use the notation from Section~\ref{sec:iso_finite}. For even $q$ the statements and proofs are almost identical to those for odd 
$q$. The necessary changes are caused by the fact that 
\[
 g_{\frac{q}2} = g_{-\frac{q}2},
\]
and that the attracting fixed point of $g_{q/2}^{-1}$ is $1$. These two properties are related to the fact that the two $\wt\Gamma$-conjugacy classes $[g_\frac{q}{2}]_{\wt\Gamma}$ and $[Qg_{\frac{q}2}]_{\wt\Gamma}$ are both related to the primitive periodic billiard on $\wt\Gamma\backslash\h$ that is represented by the geodesic on $\h$ from $-1$ to $1$, cf.\@ Section~\ref{sec:SZF}. 

For the transfer operators, $1$ being an attracting fixed point of a hyperbolic element has the effect that $1$ needs to be in the domain of definition of the functions on which the transfer operators act. Therefore, compared to the case of $q$ odd, the domains are larger. We refer to the formulas in the following, and to \cite{Pohl_spectral_hecke} for a more detailed explanation. 

In order to provide explicit formulas for the transfer operators, we note that for even $q$ we have
\[
 m = \frac{q}2 - 1.
\]
The odd and even slow transfer operator $\TO^{\slow, \pm}_{s}$ of 
$\Gamma_q$ is given by
\begin{align*}
\TO^{\slow, \pm}_{s} & =  \tfrac12 \alpha_s(g_{q/2}) \pm \tfrac12 
\alpha_s(Qg_{q/2}) + \sum_{k=1}^{m} \alpha_s(g_{-k}) \pm 
\alpha_s(Qg_{-k}) 
 \\
& =  \big( 1 \pm \alpha_s(Q)\big) \Bigg( \tfrac12 \alpha_s(g_{q/2}) + 
\sum_{k=1}^{m} \alpha_s(g_{-k})\Bigg),
\end{align*}
respectively. We consider it to act on $C^\omega( (0,1+\eps); V)$ for some $\eps > 0$ (or 
equivalently, on $C^\omega(\R_{>0};V)$). Likewise, the spaces 
$\SFE_{s}^{\omega,\pm}$, $\SFE_{s}^{\omega,\hol,\pm}$, $\SFE_{s}^{\omega,\as,\pm}$ and $\SFE_{s}^{\omega,\dec,\pm}$ are defined for functions in $C^\omega( 
(0,1+\eps);V)$.

For the odd and even fast transfer operators we need to use
\[
 \TO_{0,s}^\fast  \sceq \tfrac12\alpha_s(g_{q/2})+ \sum_{k=2}^{m} \alpha_s( 
g_{-k} ),
\]
and set
\[
 D_0 \sceq \left(\tfrac{1}{\ell} ,1\right].
\]

With these changes the statement and proof of  Theorem~\ref{thm:main_finite} 
applies for even $q$ as well.

\subsection{Isomorphism for the Theta group}\label{sec:Theta}

For the Theta group
\[
 \Gamma \sceq \Gamma_2
\]
we consider the slow and fast transfer operators that are developed in \cite{Pohl_representation}. Let 
\[
 k_1 \sceq \bmat{1}{2}{0}{1}\quad\text{and}\quad k_2 = \bmat{2}{1}{-1}{0}.
\]
In \cite{Pohl_representation}, only the full slow transfer operator for $\Gamma$ is stated, not the odd and even ones. The odd and even transfer operator are deduced by conjugating the transfer operator in \cite[Section~4.2, The reduced system]{Pohl_representation} with
\[
 \frac{1}{\sqrt{2}}\mat{\id}{\alpha_s(J)}{-\alpha_s(J)}{\id}.
\]
This conjugation provides a diagonalization of the transfer operator. The two diagonal entries are then the odd and even transfer operator.

Thus, the even (`$+$') and odd (`$-$') slow transfer operator for $\Gamma$ is (after an obvious normalizing conjugation)
\[
 \TO_s^{\slow,\pm} = \alpha_s(k_1^{-1}) + \alpha_s(k_2) \pm \alpha_s(k_2J).
\]
Both transfer operators are acting on $C^\omega( (-1,\infty);V)$. We let
\[
 \SFE_s^{\omega,\pm} \sceq \left\{\varphi\in C^\omega\big( (-1,\infty); V\big)    \ \left\vert\  \varphi = \TO_s^{\slow,\pm}\varphi  \right.\right\}
\]
be the space of real-analytic eigenfunctions with eigenvalue $1$ of $\TO_s^{\slow,\pm}$. 

Let $a\in\R$. We call a complex neighborhood $\mc U$ of the interval $(a,\infty)$ \textit{rounded at $\infty$} if there exists $x_0\in\R$ such that 
\[
 \{ z\in\C\mid \Rea z>x_0\} \subseteq \mc U.
\]
Let $\SFE_s^{\omega,\hol,\pm}$ denote the subspace of functions $\varphi\in\SFE_s^{\omega,\pm}$ that extend holomorphically to a complex neighborhood of $(-1,\infty)$ that is rounded at $-1$ and at $\infty$, and whose extensions satisfy 
\[
 f=\big(\alpha_s(k_1^{-1}) + \alpha_s(k_2) \pm \alpha_s(k_2J)\big)f
\]
on all of $\mc U$. Further, we let $\SFE_s^{\omega,\as,\pm}$ be the subspace of functions $\varphi\in \SFE_s^{\omega,\hol,\pm}$ such that there exist $c_1,c_2\in V$ (depending on $\varphi$) such that
\[
 \varphi(x) = c_1 x^{1-2s} + O_{x\to\infty}(x^{-2s}) \quad\text{and}\quad \varphi(x) = \frac{c_2}{x+1} + O_{x\to -1^+}(1).
\]
Finally, we define $\SFE_s^{\omega,\dec,\pm}$ to be the space of the functions $\varphi\in\SFE_s^{\omega,\pm}$ for which the map
\[
\begin{cases}
\varphi \pm \alpha_s(Q) \varphi & \text{on $(0,\infty)$}
\\
-\alpha_s(S)\varphi \mp \alpha_s(J)\varphi & \text{on $(-\infty,0)$}
\end{cases}
\]
extends smoothly to $\R$, and the map
\[
\begin{cases}
\varphi & \text{on $(-1,\infty)$}
\\
\mp \alpha_s(T^{-1}J)\varphi & \text{on $(-\infty,-1)$}
\end{cases}
\] 
extends smoothly to $P^1(\R)$.

In order to state the even and odd fast transfer operators for $\Gamma$ let
\[
 E_a\sceq (-1,0),\quad E_b\sceq (0,1),\quad E_c\sceq (1,\infty).
\]
Further, for $\Rea s >\tfrac12$, we set 
\[
 \TO_{1,s}^{\fast} \sceq \sum_{n\in\N} \alpha_s(k_1^{-n}), \qquad  \TO_{2,s}^\fast \sceq \sum_{n\in\N} \alpha_s(k_2^n).
\]
As for the slow transfer operator, in \cite{Pohl_representation} only the full fast transfer operator for $\Gamma$ is given explicitly. The transfer operator in \cite[Section~5.2]{Pohl_representation} can be diagonalized by the conjugation with
\[
\frac1{\sqrt{2}}
\begin{pmatrix}
1 &  & & & & \alpha_s(J)
\\
& 1 & & & \alpha_s(J)
\\
&  & 1 & \alpha_s(J)
\\
& & -\alpha_s(J) & 1
\\
& -\alpha_s(J) & & & 1
\\
-\alpha_s(J) & & & & & 1
\end{pmatrix}
\]
The even and odd transfer operators are then given by the diagonal terms. After rearranging the order of the Banach spaces and an additional normalizing conjugation, 
for $\Rea s > \tfrac12$, the even (`$+$') and odd (`$-$') fast transfer operator is given by
\[
\TO_s^{\fast,\pm} =
\begin{pmatrix}
0 & \pm\alpha_s(k_2J) & \TO_{1,s}^\fast
\\
\TO_{2,s}^\fast & \pm\alpha_s(k_2J) & \TO_{1,s}^\fast
\\
\TO_{2,s}^\fast & \pm\alpha_s(k_2J) & 0
\end{pmatrix}.
\]
Both transfer operators act on the Banach space
\[
 \mc B \sceq \mc B(E_a) \oplus \mc B(E_b) \oplus \mc B(E_c).
\]
For $\Rea s\leq \tfrac12$, these transfer operators and their components are given by meromorphic continuation.

Let 
\[
 \FFE_s^{\pm} \sceq \left\{ f\in\mc B \ \left\vert\ f=\TO_s^{\fast,\pm}f \right.\right\}
\]
and let $\FFE_s^{\dec,\pm}$ be its subspace of functions $f=(f_a,f_b,f_c)^\top\in\FFE_s^{\pm}$ such that 
\[
\begin{cases}
f_b \pm\alpha_s(Q)\left(1+\TO_{1,s}^\fast\right)f_c & \text{on $(0,1)$}
\\[1mm]
-\alpha_s(S)\left(1+\TO_{1,s}^\fast\right)f_c \mp\alpha_s(J)f_b & \text{on $(-1,0)$}
\end{cases}
\]
extends smoothly to $(-1,1)$,
\[
\begin{cases}
\left(1+\TO_{2,s}^\fast\right)f_a & \text{on $(-1,0)$}
\\[1mm]
\mp\alpha_s(T^{-1}J)f_b & \text{on $(-2,-1)$}
\end{cases}
\]
extends smoothly to $(-2,0)$, and
\[
\begin{cases}
\alpha_s(S)\left(1+\TO_{1,s}^\fast\right)f_c & \text{on $(-1,0)$}
\\[1mm]
\mp\alpha_s(ST^{-1}J)\left(1+\TO_{1,s}^\fast\right)f_c & \text{on $(0,1)$}
\end{cases}
\]
extends smoothly to $(-1,1)$.

The proof of the following theorem is analogous to that of Theorem~\ref{thm:main_finite}.

\begin{thm}\label{thm:main_theta}
Let $s\in\C\setminus\{\tfrac12\}$ with $\Rea s > 0$. Then the spaces $\SFE_s^{\omega,\as,\pm}$ and $\FFE_s^{\pm}$ are isomorphic as vector spaces. The isomorphism is given by 
\[
 \FFE_s^\pm\to \SFE_s^{\omega,\as,\pm},\quad f=(f_a,f_b,f_c)^\top \mapsto \varphi,
\]
where
\[
 \varphi\vert_{E_a} \sceq \left(1+\TO_{2,s}^\fast\right)f_a\vert_{E_a},\quad \varphi\vert_{E_b}\sceq f_b\vert_{E_b} \quad\text{and}\quad \varphi\vert_{E_c} \sceq \left( 1+\TO_{1,s}^\fast\right)f_c\vert_{E_c}.
\]
The inverse isomorphism is 
\[
 \SFE_s^{\omega,\as,\pm} \to \FFE_s^\pm, \quad \varphi\mapsto f=(f_a,f_b,f_c)^\top,
\]
where $f$ is determined by
\[
 f_a\vert_{E_a} \sceq \big(1-\alpha_s(k_2)\big)\varphi\vert_{E_a},\quad f_b\vert_{E_b}\sceq \varphi\vert_{E_b} \quad\text{and}\quad f_c \sceq \big(1-\alpha_s(k_1^{-1})\big)\varphi\vert_{E_c}.
\]
These isomorphisms induce isomorphisms between $\SFE_s^{\omega,\dec,\pm}$ and $\FFE_s^{\dec,\pm}$.
\end{thm}

\subsection{Isomorphism for non-cofinite Hecke triangle groups}\label{sec:iso_nonco}

Let 
\[
 \Gamma \sceq \Gamma_\ell
\]
be a Hecke triangle group with parameter $\ell>2$, thus a non-cofinite Fuchsian group. We consider the slow and fast transfer operators from \cite{Pohl_hecke_infinite, Pohl_representation}. To improve readibility we omit the dependence on $\ell$ in the notation.

Let
\[
 a_1 \sceq \bmat{1}{\ell}{0}{1} \quad\text{and}\quad a_2 \sceq \bmat{\ell}{1}{-1}{0}.
\]
The even and odd slow transfer operator for $\Gamma$ is given by
\[
\TO_{s}^{\slow, \pm} = \alpha_s(a_2) + \alpha_s(a_1^{-1})  \pm \alpha_s(a_2J), 
\]
acting on $C^\omega((-1,\infty);V)$. We define
\[
 \SFE_s^{\omega,\pm} \sceq \left\{ \varphi\in C^\omega\big( (-1,\infty); V\big) \ \left\vert\ \varphi = \TO_{s}^{\slow, \pm}\varphi \right.\right\}
\]
to be the space of real-analytic eigenfunctions with eigenvalue $1$ of $\TO_{s}^{\slow, \pm}$. Let $\SFE_s^{\omega,\hol,\pm}$ be its subspace of functions $\varphi\in\SFE_s^{\omega,\pm}$ that extend holomorphically to a complex neighborhood of $(-1,\infty)$ rounded at $\infty$ and whose extension satisfy the functional equation
\[
 f=\big(\alpha_s(a_2) + \alpha_s(a_1^{-1})  \pm \alpha_s(a_2J)\big)f.
\]
Further let
\[
 \SFE_s^{\omega,\as,\pm}\sceq \left\{ \varphi\in \SFE_s^{\omega,\hol,\pm}\ \left\vert\ \exists\, c\in V\colon \varphi(x) = cx^{1-2s} + O_{x\to\infty}(x^{-2s}) \vphantom{\SFE_s^{\omega,\pm}}\right.\right\}.
\]
In order to state the fast even and odd transfer operator we set
\[
E_1\sceq (-1,1)\quad\text{and}\quad E_2\sceq (\ell-1,\infty).
\]
For $\Rea s > \tfrac12$ we define
\[
 \TO_{1,s}^{\fast} \sceq \sum_{n\in\N}\alpha_s(a_1^{-n}).
\]
Then the fast even and odd transfer operator is (for $\Rea s > \tfrac12$)
\[
\TO_{s}^{\fast,\pm} =
\begin{pmatrix}
\alpha_s(a_2) \pm \alpha_s(a_2J) & \TO_{1,s}^\fast
\\[1mm]
\alpha_s(a_2) \pm \alpha_s(a_2J) & 0
\end{pmatrix},
\]
which acts on the Banach space
\[
 \mc B\sceq \mc B(E_1) \oplus \mc B(E_2).
\]
For $\Rea s \leq \tfrac12$, these transfer operators and their components are defined by meromorphic continuation. Let
\[
 \FFE_s^\pm \sceq \left\{ f \in \mc B \ \left\vert\  f= \TO_{s}^{\fast,\pm}f \right.\right\}.
\]

The proof of the following theorem is analogous to that of Theorem~\ref{thm:main_finite}.

\begin{thm}\label{thm:main_nonco}
Let $s\in\C\setminus\{\tfrac12\}$ with $\Rea s > 0$. Then the spaces $\SFE_{s}^{\omega,\as,\pm}$ and
$\FFE_{s}^\pm$ are isomorphic as vector spaces. The isomorphism is given by 
\[
 \FFE_{s}^\pm \to \SFE_{s}^{\omega,\as,\pm},\quad f=(f_1,f_2)^\top \mapsto \varphi,
\]
where
\[
 \varphi\vert_{(-1,1)} \sceq f_1\vert_{(-1,1)}\quad\text{and}\quad \varphi\vert_{(-1+\ell,\infty)} \sceq \left(1+\TO_{1,s}^\fast\right)f_2\vert_{(-1+\ell,\infty)}.
\]
The inverse isomorphism is 
\[
 \SFE_{s}^{\omega,\as,\pm} \to \FFE_{s}^\pm,\quad \varphi\mapsto f=(f_1,f_2)^\top,
\]
where $f$ is determined by
\[
 f_1\vert_{(-1,1)} \sceq \varphi\vert_{(-1,1)} \quad\text{and}\quad f_2\vert_{(-1+\ell,\infty)}\sceq  \big(1-\alpha_s(a_1^{-1})\big)\varphi\vert_{(-1+\ell,\infty)}.
\]
\end{thm}

\section{A few remarks}\label{remarks}

\begin{enumerate}[(a)]
\item The explicit formulas for the isomorphism maps in Theorems~\ref{thm:main_finite}, \ref{thm:main_theta} and \ref{thm:main_nonco} clearly show that these maps are compatible with those additional conditions on the eigenfunctions that can be expressed in similar ways for the spaces $\FFE^\pm_s$ and $\SFE^{\omega,\as,\pm}_s$. For example, every additional condition that can the expressed in terms of acting elements will result in an equivariance for the isomorphism maps.

Indeed, Theorems~\ref{thm:main_finite}, \ref{thm:main_theta} and \ref{thm:main_nonco} are themselves examples for the latter observation if we use Theorem~A as a starting point. To be more precise, let 
\[
 \FFE_s \sceq \{ f \mid f=\TO_s^\fast f\} \quad\text{and}\quad \SFE_s \sceq \{ f\mid f=\TO_s^\slow f\}
\]
be the eigenspaces with eigenvalue $1$ of $\TO_s^\fast$ and $\TO_s^\slow$, respectively. Since we only intend to provide a sketch for the mentioned compatibility of the isomorphism map with certain symmetries, we do not discuss the necessary regularity properties of the eigenfunctions. Let 
\[
 \Phi\colon \FFE_s \to \SFE_s
\]
be the isomorphism map in Theorem~A that is constructed analogously to the isomorphism maps in Theorem~B (see Theorems~\ref{thm:main_finite}, \ref{thm:main_theta} and \ref{thm:main_nonco}). We did not provide a separate formula for the isomorphism map $\Phi$. However, $\Phi$ is essentially the pair of the two isomorphism maps in Theorem~B (consisting of the isomorphism maps $\FFE^+_s\to \SFE^{\omega,\as,+}_s$ and $\FFE^-_s\to\SFE^{\omega,\as,-}_s$).

The even and odd eigenfunctions of the fast transfer operator $\TO_s^\fast$ are then detected by invariance and anti-invariance under $\alpha_s(Q)$, respectively, and likewise for the slow transfer operator $\TO_s^\slow$. The isomorphism map $\Phi$ is $\alpha_s(Q)$-equivariant. Theorem~B or, more precisely, Theorems~\ref{thm:main_finite}, \ref{thm:main_theta} and \ref{thm:main_nonco} (for the latter, using $\alpha_s(J)$ instead of $\alpha_s(Q)$) state the already refined isomorphisms between the spaces of even or odd eigenfunctions.

We leave the investigation of further additional conditions for future work. Examples that should be considered include other exterior symmetries such as, e.\,g., Hecke operators. Also other types of conditions, e.\,g., fixed values at common fixed points, are of interest.

\item\label{remb} Patterson conjectured a relation between the divisors of Selberg zeta functions and certain cohomology spaces \cite{Patterson_israel} (see also \cite{Bunke_Olbrich_annals, Juhl_book, Deitmar_Hilgert_cohom}). For Fuchsian lattices $\Gamma$, Bruggeman, Lewis and Zagier provided a characterization of the space of Maass cusp forms for $\Gamma$ with spectral parameter $s$ as the space of parabolic $1$-cohomology with values in the semi-analytic, smooth vectors of the principal series representation for the parameter $s$ \cite{BLZ_part2}. In connection with the Selberg trace formula, these results support Patterson's conjecture.

In \cite{Moeller_Pohl, Pohl_spectral_hecke, Pohl_representation} the second author (for $\Gamma_\ell$ with $\ell<2$  jointly with M\"oller) established an (explicit) isomorphism between $\SFE_s^{\omega,\dec,\pm}$ and the corresponding cohomology space from \cite{BLZ_part2}. In turn, Theorems~A and B support Patterson's conjecture within a transfer operator framework (and without using the Selberg trace formula).

We stress that the relation which arises from the transfer operator techniques between those spectral zeros of the Selberg zeta function which are spectral parameters of Maass cusp forms and the (dimension of the) cohomology spaces  is canonical. In particular, this relation does not depend on the choice of an admissible discretization for the geodesic flow. 

It would be interesting to see if there is also such a cohomological interpretation of $\SFE_s^{\omega,\as,\pm}$ if $s$ is not a spectral parameter of a Maass cusp form. Moreover, it would be desirable to find an extension of such a cohomological framework which allows to include non-trivial representations as well as non-cofinite Fuchsian groups.

\item At the state of art, Theorem~\ref{slow_props} and its generalizations \cite{Pohl_mcf_general, Moeller_Pohl, Pohl_mcf_Gamma0p} are restricted to cofinite Fuchsian groups. 

However, Patterson showed that also for non-cofinite Hecke triangle groups $\Gamma$ there is at least one (normalized) $L^2$-eigenfunction of the Laplace--Beltrami operator with spectral parameter $\delta = \dim_H \Lambda(\Gamma)$ being the Hausdorff dimension of the limit set $\Lambda(\Gamma)$ of $\Gamma$ \cite{Patterson}. 
Moreover, Lax and Phillips investigated the spectral theory of the Laplacian on hyperbolic manifolds of any dimension \cite{Lax_Phillips_latticepoints, Lax_Phillips_translationI, Lax_Phillips_translationII, Lax_Phillips_translationIII}. For non-cofinite Hecke triangle groups, these results show that all $L^2$-eigenvalues of the Laplace--Beltrami operator are contained in the interval $(0, 1/4)$, and that there is at least one. In particular, there are no $L^2$-eigenvalues embedded into the absolutely continuous spectrum. 

By these spectral results, it is reasonable to expect that an analogue of Theorem~\ref{slow_props} is valid for non-cofinite Hecke triangle groups as well. 
\item It is expected that analogues of Theorem~\ref{slow_props} can be shown for $\Rea s\notin (0,1)$, $\chi$ any finite-dimensional unitary representation, and $\Gamma$ cofinite or non-cofinite, see \cite[Section~7]{Pohl_hecke_infinite}, \cite[Conjectures~4.2, 4.6]{Pohl_representation}. In this case, the role of Maass cusp forms is expected to be played (in some way) by $\chi$-twisted resonant states.
\item It would further be desirable to characterize the elements in $\SFE_s^{\omega,\as,\pm}$ that are not contained in $\SFE_s^{\omega,\dec,\pm}$ purely in a transfer operator framework (in particular, without relying on the Selberg trace formula). A complete characterization would allow us to provide---as a by-product, and independent of the Selberg trace formula---a complete classification of the zeros of the Selberg zeta function. For the case that $\Gamma$ is the modular group $\PSL_2(\Z)$ and $\chi$ is the trivial one-dimensional representation, the combination of \cite{Bruggeman, Chang_Mayer_transop, Chang_Mayer_period, Lewis_Zagier, Deitmar_Hilgert} provides such characterizations.
\end{enumerate}

\appendix

\section{Odd and even Selberg zeta functions}\label{sec:SZF_proof}

Recall the Selberg zeta functions $Z$ and $Z_\pm$ from Section~\ref{sec:SZF}. In this section we provide a sketch of the proof that $Z=Z_+\cdot Z_-$. 

\begin{lemma}
For all Hecke triangle groups $\Gamma$ and all finite-dimensional unitary representations $\chi$ we have $Z=Z_+\cdot Z_-$.
\end{lemma}

\begin{proof}[Sketch of proof]
It suffices to show the equality
\[
 Z(s) = Z_+(s)Z_-(s)
\]
for those $s\in\C$ for which $Z$ and $Z_\pm$ are given by the infinite products from Section~\ref{sec:SZF}. Equality on $\C\setminus\{\text{poles}\}$ then follows from meromorphic continuation.

Let $\Gamma$ be a Hecke triangle group, and let $\TO_s^\fast$, $\TO_s^{\fast,\pm}$ denote its fast transfer operators. From \cite{Moeller_Pohl, Pohl_hecke_infinite, Pohl_representation} it is known (or easily deduced) that 
\[
 Z(s) = \det\left(1-\TO_s^\fast\right)
\]
and 
\[
 \det\left(1-\TO_s^\fast\right) = \det\left(1-\TO_s^{\fast,+}\right)\det\left(1-\TO_s^{\fast,-}\right).
\]
Thus, it suffices to show 
\begin{equation}\label{zetatoshow}
 Z_\pm(s) = \det\left(1-\TO_s^{\fast,\pm}\right).
\end{equation}

We provide a sketch of the proof of \eqref{zetatoshow} only for the Hecke triangle groups $\Gamma_\ell$ with $\ell=2\cos\tfrac{\pi}q$, $q\in\N_{\geq 4}$ even, and the odd transfer operator. All remaining instances of \eqref{zetatoshow} are shown analogously. 

Let $\Gamma\sceq \Gamma_\ell$ and set $\wt\Gamma\sceq \langle Q,\Gamma\rangle$. Recall $g_\mu$ from \eqref{elemmu} and $[\wt\Gamma]_{p,\mu}$ from \eqref{setmu}. We assign to $\mu$ the numerical value 
\[
 \mu\sceq \frac{q}2.
\]
Note that definitions \eqref{elemmu} and \eqref{defgqk} remain consistent, and $g_\mu = g_{-\mu}$.

Let 
\[
 \Gen \sceq \{ g_{-2},\ldots, g_{-\mu}, Qg_{-2},\ldots, Qg_{-\mu}\} \cup \{ g_{-1}^k, Qg_{-1}^k \mid k\in\N\}.
\]
For $h=h_1\ldots h_n$ with $n\in\N$ and $h_j\in \Gen$, $1\leq j\leq n$, let 
\[
 b_s^-(h) \sceq \frac{(-1)^\eps}{2^k}\alpha_s(h),
\]
where 
\[
 \eps \sceq \eps(h) \sceq \#\left\{ j\in\{1,\ldots, n\} \ \left\vert\  h_j\in\left\{ Qg_{-2},\ldots, Qg_{-\mu}, Qg_{-1}^\ell\ \left\vert\  \ell \in\N \vphantom{Qg_{-1}^\ell} \right.\right\}  \right.\right\}
\]
and
\[
 k\sceq k(h)\sceq \#\{ j\in\{1,\ldots, n\} \mid h_j\in\{g_\mu, Qg_\mu\}\}.
\]
We consider $h$ as a word of length $n$ over $\Gen$ and call $h$ \textit{reduced} if it does not contain a subword of the form $g_{-1}^{m_1}g_{-1}^{m_2}$ or $Qg_{-1}^{m_1}g_{-1}^{m_2}$ with $m_1,m_2\in\N$. We let $W_n^\redu(\Gen)$ denote the set all of reduced words over $\Gen$ of length $n$, and define
\[
 W_*^\redu(\Gen) \sceq \bigcup_{n\in\N} W^\redu_n(\Gen).
\]
For $n\in\N$ we let $C_1^n$ be the subset of words in $W^\redu_n(\Gen)$ that end with $g^\ell_{-1}$ or $Qg^\ell_{-1}$ for some $\ell\in\N$ and do not begin with $g_{-1}^k$ for any $k\in\N$. Further we let $C_2^n$ be the subset of words in $W^\redu_n(\Gen)$ that end with an element of $\{ g_k, Qg_k\mid k\in\{-\mu,\ldots, -2\}\}$.

By \cite[Lemma~6.2]{Pohl_spectral_hecke}, 
\[
 \left(\TO_s^{\fast,-}\right)^n = \mat{\sum\limits_{a\in C_1^n} b_s^-(a)}{\ast}{\ast}{\sum\limits_{a\in C_2^n}b_s^-(a)}.
\]
The off-diagonal entries are known as well but are not of importance for our applications. For all $a\in C_1^n\cup C_2^n$ we have (combine \cite[6.4]{Pohl_spectral_hecke} and \cite[Lemma~5.2]{Pohl_representation})
\[
 \Tr b_s^-(a) = \frac{\det a}{2^{k(a)}} \frac{N(a)^{-s}}{1-\det a\cdot N(a)^{-1}} \tr\chi(a).
\]

Let 
\[
 Z_\reg(s) \sceq \prod_{[g]_{\wt\Gamma}\in [\wt\Gamma]_{p,\mu}} \prod_{k=0}^\infty \det\left( 1 - \det g^{k+1}\cdot \chi(g) N(g)^{-(s+k)}\right).
\]
Then 
\begin{equation}\label{logzeta}
 \log Z_-(s) = \log Z_\reg(s) + \log Z_{\mu,\id} - \log Z_{\mu, Q}.
\end{equation}
Let 
\[
 [\wt\Gamma]_{h,\mu} \sceq \left\{ [g^n]_{\wt\Gamma} \ \left\vert\  [g]_{\wt\Gamma}\in [\wt\Gamma]_{p,\mu},\ n\in\N \vphantom{[g^n]_{\wt\Gamma}}  \right.\right\}.
\]
Using \cite[Proof of Theorem~6.1]{Pohl_spectral_hecke} with the extension to unitary representations as in \cite{Pohl_representation} we see that 
\begin{equation}\label{part1zeta}
 \log Z_\reg(s) = -\sum_{n\in\N} \frac1n \sum_{\substack{a\in C_1^n\cup C_2^n \\ [a]_{\wt\Gamma}\in [\wt\Gamma]_{h,\mu}}} \Tr b_s^-(a).
\end{equation}
In order to relate the other summands in \eqref{logzeta} to the traces of the transfer operator we let $W_\mu\sceq W_*^\redu(\{g_\mu, Qg_\mu\})$ denote the set of words over the alphabet $\{g_\mu, Qg_\mu\}$. Note that each element in $[\wt\Gamma]\smallsetminus [\wt\Gamma]_{h,\mu}$ has a representative in $W_\mu$.

Analogously to \cite[Proof of Theorem~6.1]{Pohl_spectral_hecke}  we find
\begin{align}
-\sum_{p=1}^\infty & \frac1p  \sum_{ \substack{a\in C_1^p\cup C_2^p \\ a\in W_\mu}   }\label{det0}
 \Tr b_s^-(a)  = \sum_{p=1}^\infty \frac{N(g_\mu^p)^{-s}}{2p} \left( \frac{\tr\chi(Qg_\mu^p)}{1+ N(g_\mu^p)^{-1}}  - \frac{ \tr\chi(g_\mu^p)}{1-N(g_\mu^p)^{-1}}\right) 
 \\
 & = \sum_{p=1}^\infty \frac{1}{2p} \frac{1}{1-N(g_\mu^p)^{-2}} \Big[ N(g_\mu^p)^{-s}\tr\chi(Qg_\mu^p) - N(g_\mu^p)^{-(s+1)} \tr\chi(Qg_\mu^p) \nonumber
 \\
 & \hphantom{\sum_{p=1}^\infty \frac{1}{2p} \frac{1}{1-N(g_\mu^p)^{-2}}}\qquad - N(g_\mu^p)^{-s} \tr\chi(g_\mu^p) - N(g_\mu^p)^{-(s+1)}\tr\chi(g_\mu^p)  \Big]. \nonumber
\end{align}
Further, we have
\begin{align}
\sum_{p=1}^\infty \frac{1}{2p} \frac{N(g_\mu^p)^{-s}\tr\chi(Qg_\mu^p)}{1-N(g_\mu^p)^{-2}} & = \frac12\sum_{k=0}^\infty \sum_{p=1}^\infty \frac1p N(g_\mu^p)^{-(s+2k)}\tr\chi(g_\mu^pQ) \nonumber
\\
& = \frac12\sum_{k=0}^\infty \tr\left( \sum_{p=1}^\infty \frac1p N(g_\mu^p)^{-(s+2k)}\chi(g_\mu^p)\chi(Q)\right)\nonumber
\\
& = -\frac12\log\det\exp\left( \log\left(1-\chi(g_\mu)N(g_\mu)^{-(s+2k)}\right)  \cdot\chi(Q)\right)\nonumber
\\
& = -\frac12 \log\prod_{k=0}^\infty \det\left( \left(1-\chi(g_\mu)N(g_\mu)^{-(s+2k)}\right)^{\chi(Q)} \right). \label{det1}
\end{align}
Analogously, we find
\begin{align}\label{det2}
 \sum_{p=1}^\infty \frac{1}{2p} \frac{N(g_\mu^p)^{-s}\tr\chi(g_\mu^p)}{1-N(g_\mu^p)^{-2}} = -\frac12 \log\prod_{k=0}^\infty \det\left(1-\chi(g_\mu)N(g_\mu)^{-(s+2k)}\right).
\end{align}
Using \eqref{det1} and \eqref{det2} in \eqref{det0} and comparing to \eqref{logzeta} shows that 
\[
 \log Z_{\mu,\id} - \log Z_{\mu, Q} = -\sum_{p\in\N} \frac1p \sum_{\substack{a\in C_1^p\cup C_2^p \\ a\in W_\mu}} \Tr b_s^-(a).
\]
In combination with \eqref{part1zeta} this completes the sketch of the proof.
\end{proof}

\bibliography{ap_bib}
\bibliographystyle{amsplain}

\end{document}